\documentclass[11pt]{amsart}

\usepackage{amsmath,amsthm,amsfonts,amssymb, mathrsfs}
\usepackage{parskip} 
\usepackage[margin=1in]{geometry}
\usepackage[shortlabels]{enumitem}
\usepackage{color}
\usepackage{hyperref}

\calclayout
\numberwithin{equation}{section}

\newtheorem{thm}{Theorem}[section]
\newtheorem{lem}{Lemma}[section]
\newtheorem{prop}{Proposition}[section]

\newtheorem{defn}{Definition}[section]
\theoremstyle{definition}
\newtheorem{assum}{Assumption}[section]
\newtheorem{rem}{Remark}[section]
\theoremstyle{definition}

\newcommand{\Cov}{\text{Cov}}

\numberwithin{equation}{section}
\newcommand{\be}{\begin{equation}}
	\newcommand{\ee}{\end{equation}}
\newcommand{\bes}{\begin{equation*}}
	\newcommand{\ees}{\end{equation*}}

\newcommand{\bP}{\mathbb{P}}

\newcommand{\R}{\mathbf{R}}

\newcommand{\Z}{\mathbf{Z}}

\newcommand{\N}{\mathbf{N}}

\newcommand{\E}{\mathbb{E}}

\newcommand{\bean}{\begin{eqnarray*}}
	\newcommand{\eean}{\end{eqnarray*}}

\newcommand{\ve}{\varepsilon}

\newcommand{\K}{\mathbf{K}}

\title[Blowup for multiplicative SHE]{Blowup for the multiplicative stochastic heat equation with superlinear drift}
\author{Mathew Joseph \and Shubham Ovhal}
\thanks{}

\address{Mathew Joseph\\ Statmath Unit\\ Indian Statistical Institute\\ 8th Mile Mysore Road\\ Bangalore 560059
} \email{m.joseph@isibang.ac.in}
\address{Shubham Ovhal\\ Statmath Unit\\ Indian Statistical Institute\\ 8th Mile Mysore Road\\ Bangalore 560059
} \email{rs\_math2104@isibang.ac.in}

\begin{document}
\maketitle

\begin{abstract}
We consider the stochastic heat equation with multiplicative white noise: $\partial_t u =\partial_x^2u + b(u) +\sigma(u) \dot W$, both on $[0,1]$ and $\R$. In the case of $[0,1]$ we show that the finite Osgood criterion on $b$ is a necessary and sufficient condition for finite-time blowup, under fairly general conditions on $\sigma$. In the case of $\R$ we show instantaneous explosion when we start with initial profile $u_0\equiv 1$. Our work introduces a novel approach that not only extends the work of \cite{foon-khos-nual} which dealt with bounded $\sigma$, but also provides an alternate route to proving such results. A key ingredient is a comparison result which shows that the solution on $\R$ stays above the corresponding solution on $[0,1]$ with Dirichlet boundary conditions. 


\end{abstract}
\section{Introduction}
We consider the solution $v_t(x),\, t\ge 0, \, x\in [0,1]$ of the following stochastic heat equation
\be \tag{SHE:D} \label{eq:SHE:D}
    \partial_t v_t(x) = \frac{1}{2} \partial_x^2 v_t(x) + b\left(v_t(x)\right) + \sigma\left(v_t(x)\right) \dot W(t, x),\qquad t\ge 0, \, x\in [0,1],
\ee
with {\it Dirichlet} boundary conditions. The initial profile $v_0$ is nonnegative and bounded. The noise $\dot W(t,x)$ is space-time white noise: the centered, generalized Gaussian random field with covariance
\bes
\Cov[\dot W(t, x), \dot W(s, y)] = \delta(t-s) \delta(x-y).
\ees
The function $\sigma: \R \to \R$ is {\it locally} Lipschitz and satisfies $\sigma(0)=0$. The function $b: \R \to \R_{\ge 0}$ is nonnegative, {\it non-decreasing} and {\it locally} Lipschitz. 

Let $p_t(x,y)$ be the Dirichlet heat kernel. We say $v$ is a {\it local solution} to \eqref{eq:SHE:D} if 
\[ v_t(x) = (p_t*v_0)(x) + \int_0^t \int_0^1 p_{t-s}(x,y) b(v_s(y)) \, dy ds + \int_0^t \int_0^1 p_{t-s}(x,y) \sigma(v_s(y)) \, W(dy ds),\]
almost surely for all $t< \tau_J:= \inf\{t> 0: \|v_t\|_{\infty} \ge J\}$, and all $J>\|v_0\|_{\infty}$. Moreover, we also require $\tau_J>0$ almost surely. We call 
\[ \tau_{\infty}:= \sup_J \tau_J\] 
the {\it blowup time} of $v$. If $\tau_{\infty}= \infty$ almost surely, we say that a {\it global solution} exists (see \cite{sali}).


As $b$ and $\sigma$ are locally Lipschitz, we have existence and uniqueness of local solutions \cite{spde-utah, wals}. Moreover, the solutions remain nonnegative \cite{muel}. We will make a few more additional assumptions on $b$ and $\sigma$. Let $f(x)$ be defined as:
\[ f(x) := \frac{b(x)}{x}. \]
\begin{assum}\label{ass:D} Along with the above assumptions on $b$ and $\sigma$ we make the following additional assumptions:
\begin{enumerate}
\item \label{ass:1} For all $X>0$ we have $\sup_{x\le X} |\sigma(x)| \le X g(X)$ where $g$ has the following property: There exists $\gamma>0$ such that for all $C\ge 1$
\[ \sup_{x\ge \gamma}\left|\frac{g(Cx)}{g(x)}\right| <\infty.\]
 \item \label{ass:2} There exists an $\eta>0$ such that for $x\ge 1$
\[ |g(x)| \lesssim f(x)^{\frac14-\eta}.\]
\item \label{ass:3} There exists $\mathbf C_1>1$ such that for all $X$ large enough:
\[ \int_X^{\mathbf C_1X}\frac{dx}{b(x)} > \frac{32X}{b(X/32)}.\]
\end{enumerate}
\end{assum}
The aim of this article is to show finite-time blowup under the {\it finite Osgood condition} on $b$:
\be \label{eq:osgood}
    \int_{\theta}^\infty \frac{dx}{b(x)} < \infty \quad \text{ for some } \theta>0.
\ee
We first address the blowup question for \eqref{eq:SHE:D}.
\begin{thm}\label{thm:dirichlet} Let Assumption \ref{ass:D} hold. Let $v_0$ be a nonnegative continuous function which is positive in an open interval. 
\begin{enumerate}
\item \label{state:1} Suppose \eqref{eq:osgood} holds and for each $\mu>0$ we have $\inf_{x\ge \mu} \sigma(x)>0$. Let $0<a<\frac12$. Then for each $\delta>0$ the following holds.
\be \label{eq:D:blow} \bP\left(\sup_{t\le \delta}\left(\sup_{x\in [a, 1-a]}v_t(x)\right)=\infty\right)>0.\ee
\item \label{state:2} Suppose (1.1) does not hold. Then global solutions to \eqref{eq:SHE:D} exist.
\end{enumerate}
\end{thm}
\begin{rem}The above result is also true for periodic or Neumann boundary conditions. One could also consider equation \eqref{eq:SHE:D} on an interval $[0,J]$ (instead of $[0,1]$) with Dirichlet, periodic or Neumann boundary conditions, and obtain similar results; for the Dirichlet case we would then replace $\sup_{[a,1-a]}$ by $\sup_K$ for a compact subinterval $K$ of $(0,J)$.
\end{rem}
\begin{rem}[\bf Important] \label{rem:2} Examples of drifts $b$ that satisfy Assumption \ref{ass:D}(\ref{ass:3}) are $x^{1+\eta}$ for {\it small} $\eta>0$. There are many other examples of $b$ which do not grow much faster than linearly, for example 
\begin{align*}
b(x)&= x(\log x)(\log \log x)\cdots (\underbrace{\log\log...\log}_{n \text{ times}} x)\\
b(x)&=x(\log x)(\log \log x)\cdots (\underbrace{\log\log...\log}_{(n-1) \text{ times}} x)(\underbrace{\log\log...\log}_{n \text{ times}} x)^{1+\eta}\qquad (\eta>0);\end{align*}
 the second satisfies \eqref{eq:osgood} while the first does not (see Lemma \ref{lem:b:ass}). Assumption \ref{ass:D}(\ref{ass:2}) then puts a bound on how fast $\sigma$ can grow. Note that Assumption \ref{ass:D}(\ref{ass:3}) excludes drift functions which grow very fast, e.g. $b(x)=e^{x^2}$.  A very useful result of \cite{geis-mant} states   that if $v$ (resp. $\widetilde v$) solves \eqref{eq:SHE:D} with $(b, \sigma)$ (resp. $(\widetilde b, \sigma)$) and the same initial profile, and if $b\ge \widetilde b$, then $v\ge \widetilde v$ up to the blowup time of $v$.  Therefore we will have blowup by time $\delta$ with positive probability for $(b=e^{x^2}, \sigma= x(\log x)^{\frac12-\eta})$ since we have it for $(b=x(\log x)^2, \sigma= x(\log x)^{\frac12-\eta})$. Similarly since we have global existence for $(b= x\log x, \sigma=x(\log x)^{\frac14-\eta})$ we will have global existence for $(\widetilde b, \sigma= x(\log x)^{\frac14-\eta})$ for $\widetilde b \le x\log x$. The crucial observation is that it is enough to understand blowup (resp. global existence) for drifts which {\it just} satisfy (resp. {\it just} fail to satisfy) \eqref{eq:osgood}.  
\end{rem}
\begin{rem} \label{rem:3} Statement \ref{state:2} of the above theorem is not new, and can be proved similarly to Theorem 1.6 of \cite{chen-foon-huan-sali} which considered the stochastic heat equation on $\R$. However, in the case of a bounded interval $[0,1]$ the proof turns out to be extremely simple as we will see later. Moreover Assumption \ref{ass:D}(\ref{ass:2}) is quite close to their main assumption on $b, \sigma$ (they roughly assume $|g(x)|\lesssim f(x)^{\frac14}$ times some log corrections). On the other hand, statement \ref{state:1} of the above theorem is entirely new. So far, explosion under \eqref{eq:osgood} has only been shown in the case when $\sigma$ is bounded (see \cite{foon-nual}, see Theorem 1.7 of \cite{chen-foon-huan-sali} for explosion on $\R$).  
\end{rem}

We next discuss blowup for the solution $u_t(x),\, t\ge 0, \, x\in \R$ of the following stochastic heat equation 
\be \tag{SHE} \label{eq:SHE}
    \partial_t u_t(x) = \frac{1}{2} \partial_x^2 u_t(x) + b(u_t(x)) + \sigma\left(u_t(x)\right) \dot W(t, x),\qquad t\ge 0, \, x\in \R,
\ee
with bounded nonnegative initial profile $u_0$. 
  We will continue to have the same assumptions (Assumption \ref{ass:D}) on $b$ and $\sigma$. In addition, we {\it will} assume that $\sigma$ is {\it globally} Lipschitz and that $b$ satisfies the finite Osgood criterion \eqref{eq:osgood}. 

Due to the fast growth of $b$ it is not immediately clear what we mean by a solution to \eqref{eq:SHE}. For $J>0$, let $u^{(J)}$ be the solution to 
\bes
 \partial_t u^{(J)}_t(x) = \frac{1}{2} \partial_x^2 u^{(J)}_t(x) + b\left(u^{(J)}_t(x)\wedge J\right) + \sigma\left(u^{(J)}_t(x)\right) \dot W(t, x),\qquad t\ge 0, \, x\in \R,
\ees
with initial profile $u^{(J)}_0=u_0$. As $\sigma$ is Lipschitz and $b$ is locally Lipschitz, a random field solution $u^{(J)}$ exists \cite{spde-utah, wals} and is nonnegative \cite{muel}. Moreover by a comparison theorem \cite{geis-mant} we see that $u^{(J_1)}_t(x) \le u^{(J_2)}_t(x)$ whenever $J_1\le J_2$. We define the solution to \eqref{eq:SHE} as 
\bes
u_t(x) = \lim_{J \to \infty}u_t^{(J)}(x).
\ees
See \cite{foon-khos-nual, jose-ovha} for a brief justification for why $u_t(x)$ may be labelled as a `minimal solution' of \eqref{eq:SHE}. We have the following
\begin{thm}\label{thm:line} Let Assumption \ref{ass:D} hold. In addition, suppose that $\sigma$ is globally Lipschitz and $b$ satisfies the finite Osgood criterion \eqref{eq:osgood}. Let $u_0$ be a nonnegative continuous function which is positive in an open interval around $0$. Fix any $\delta>0$ and $M>0$.
\begin{enumerate}
\item \label{2:state:1}  Suppose in addition for each $\mu>0$ we have $\inf_{x\ge \mu}\sigma(x)>0$. We have 
\bes
\bP\left(\sup_{t\le \delta} \left(\sup_{x\in [-M,\, M]} u_t(x)\right) =\infty\right)>0.
\ees 
\item \label{2:state:2} Suppose we assume the conditions of the above statement, and let $u_0\equiv 1$. We have instantaneous explosion:
\bes
\bP\left(\text{there is a point } p \in \R \text{ such that } \sup_{t\le \delta} \left(\sup_{x\in [p-M,\, p+M]} u_t(x)\right) =\infty\right)=1.
\ees
\end{enumerate}
\end{thm}

We now refine the previous theorem to an \emph{everywhere} blowup for the special case when $\sigma(x)=x$ and $b(x)=x\log^{1+\theta}(1+x),\, \theta>3$. For the solution $v$ to \eqref{eq:SHE:D}, we show a stronger notion of blowup. Indeed, we show that the $L^1$ norm blows up with positive probability. For the solution $u$ to \eqref{eq:SHE}, we improve Theorem\ref{thm:line}(\ref{2:state:2}) to instantaneous everywhere explosion. 
\begin{thm}\label{thm:every:dir}
Let $\sigma(x)=x$ and $b(x)=x\log^{1+\theta}(1+x)$ and assume $\theta>3$. Fix $\delta>0$. 
\begin{enumerate} 
\item \label{t3:p1} Let $v_0$ be a nonnegative continuous function which is positive in an open interval.
We have
\begin{align*}
    \bP\left( \sup_{t\le \delta} \int_0^1 \sin(\pi x) v_t(x) dx =\infty\right) >0.
\end{align*}
\item \label{t3:p2} Let $u_0\equiv 1$. Then we have instantaneous everywhere explosion:
\bes
\bP\left( \sup_{t\le \delta} \left(\inf_{x\in\R} u_t(x)\right) =\infty\right)=1.
\ees
\end{enumerate}
\end{thm}
\begin{rem} The proof will make it clear that the above also holds when $C_1x\log^{1+\theta_1}(1+x) \le b(x)\le C_2x\log^{1+\theta_2}(1+x)$, for $\theta_1, \theta_2>3$, and $\sigma(x)=C_3x$. 
\end{rem}
The proof of the first part of the above theorem is heavily inspired by a recent preprint of Salins and Zhang \cite{sali-zhan}.

There have been quite a few results on global existence of solutions of \eqref{eq:SHE:D} when \eqref{eq:osgood} does not hold. It was shown in \cite{dala-khos-zhan} that if $ b(u) =O(\vert u\vert  \log \vert u\vert )$ and $\vert \sigma(u)\vert  = o\left(\vert u\vert  (\log \vert u\vert )^{\frac{1}{4}}\right)$ as $\vert u\vert \to \infty$, and if the initial profile is H\"older continuous, then there exists a global solution of \eqref{eq:SHE:D}. Our ({\it much simpler}) method yields a conclusion that comes close; that is, for $b(u) =O(\vert u\vert  \log \vert u\vert )$ we show existence under the assumption $\vert \sigma(u)\vert  = O\left(\vert u\vert  (\log \vert u\vert )^{\frac{1}{4}-\eta}\right)$ for some $\eta>0$; our proof will show that we do not need to assume anything more than boundedness of the initial profile. We are also able to show global existence for drifts which grow faster than $|u|\log |u|$ with appropriate conditions on $\sigma$. We mention here that in the case $b(u) = u\log u$ the condition on $\sigma$ is not optimal; one can even take $\sigma(u)= |u|^{\gamma}(\log |u|)^{\theta}$ with $0\le \gamma<\frac32$ and $\theta\ge 0$ for global existence \cite{salins2025}. There have also been some works on $L^2$ valued solutions to \eqref{eq:SHE:D} with drift $b(u)=O(|u \log |u|)$; see \cite{dala-khos-zhan},  \cite{shan-zhan-2}, \cite{foon-khos-nual-local}. See also \cite{chen-huan}, \cite{shan-zhan} for global existence of \eqref{eq:SHE} with $b(u)=O(|u|\log |u|)$. While most of the above articles were for specific growth of $b$, recently there have been a series of papers (\cite{chen-foon-huan-sali}, \cite{sali-3}, \cite{sali-2}, \cite{sali},\cite{sali-4}, \cite{salins2025}) which show that the infinite Osgood condition implies global existence. As mentioned in Remark \ref{rem:3}, the second statement of Theorem \ref{thm:dirichlet} comes close to the global existence result of \cite{chen-foon-huan-sali}, which has the most relaxed condition on $\sigma$ under which one has global existence; our proof is much simpler for a bounded domain. 

 The first work on blowup of \eqref{eq:SHE:D} was \cite{bond-groi}, where the authors proved blowup for \eqref{eq:SHE:D} in the case when $\sigma$ is a constant and the drift $b$ satisfies \eqref{eq:osgood}. Their method was based on an application of `Feller's explosion test' for SDEs. They in fact, show the blowup of $\int_0^1 \sin(\pi x) v_t(x) dx $ as in Theorem \ref{thm:every:dir} (\ref{t3:p1}). Subsequently, \cite{foon-nual} proved the necessity of the finite Osgood condition for blowup in the case when $\sigma$ is constant; their method was based on the classical Osgood's theorem for ODEs adapted for a pathwise analysis.  The paper \cite{chen-foon-huan-sali} shows that \eqref{eq:osgood} is a necessary and sufficient condition for finite-time explosion of \eqref{eq:SHE} in the case of {\it bounded} $\sigma$. Theorem \ref{thm:dirichlet} proves the same for any $\sigma$ for which $\left|\sigma(x)/x\right| \lesssim \left(b(x)/x\right)^{\frac14-\eta}$.  In the case of \eqref{eq:SHE} on $\R$ the work \cite{foon-khos-nual} proves instantaneous everywhere blowup under the assumption that $\sigma$ is both bounded away from $0$ and $\infty$, for {\it any} initial profile. Their key step was a proof of the ergodicity of the stochastic integrals of the solution (using Malliavin calculus), and to show that the stochastic integrals can take arbitrarily large values uniformly over a space-time rectangle with positive probability. Their method does {\it not} work for unbounded $\sigma$; in particular, the important case $\sigma(u)=u$ does not work. Theorems \ref{thm:line} and \ref{thm:every:dir} extend the instantaneous explosion result of \cite{foon-khos-nual}, which dealt with bounded $\sigma$, to unbounded $\sigma$. 


Let us describe the key ideas in the proofs of the above theorems. The proof of Theorem \ref{thm:dirichlet} follows a somewhat standard approach in blowup problems in a bounded domain. Namely, one considers the evolution of $v$ between levels $2^n, \, n \in \Z$. The key insight in our proof is to control the times it takes to go between the levels. The assumptions in Assumption \ref{ass:D} allow us to control these times; we are able to show that these are of order $\frac{2^n}{b(2^n)}$ with very high probability. This sequence is summable if and only if \eqref{eq:osgood} holds. Theorem \ref{thm:line} follows from Theorem \ref{thm:dirichlet} via a key comparison result (Theorem \ref{thm:compare}) which shows that $u$ remains above $v$, if both of them start with the same initial profile $v_0$ supported in $[0,1]$. In a sense, what we show is that the blowup of \eqref{eq:SHE} is caused by local blowup, that is, blowup of \eqref{eq:SHE:D}. The reader might notice some similarity with our earlier work on blowup for interacting SDEs \cite{jose-ovha}. Theorem \ref{thm:every:dir}(\ref{t3:p1}) borrows heavily from the work \cite{sali-zhan}. The key idea is to show that as $\|v_t\|_{\infty}$ moves to higher levels $2^n$, the mass $\int_0^1\sin(\pi x) v_t(x) dx$ is at most a polynomial (in n) factor away. Theorem \ref{thm:every:dir}(\ref{t3:p2}) is derived from Theorem \ref{thm:every:dir}(\ref{t3:p1}) and Theorem \ref{thm:compare}.

Section \ref{sec:comparison} proves a key comparison result which might be of independent interest. In Section \ref{sec:dirichlet} we prove Theorem \ref{thm:dirichlet}. Theorem \ref{thm:line} is proved in Section \ref{sec:line}. We prove Theorem \ref{thm:every:dir} in Section \ref{sec:every:dir}. The Appendix collects a couple of lemmas which were used in Section \ref{sec:comparison}. Finally, acknowledgements and an AI disclosure conclude the paper.

{\bf Notation:} $a(x) \lesssim b(x)$ means there is a $C>0$ such that $a(x) \le C b(x)$ for all $x$. Constants $C$ might change from line to line. For a random variable $Z$,  $\|Z\|_p$ will denote it's $L^p$ norm. We denote $a\wedge b$ (resp. $a\vee b$) for the minimum (resp. maximum) of two real numbers $a$ and $b$.

\section{Comparison of SHE on $[0,1]$ to SHE on $\R$}\label{sec:comparison}
This section proves a key comparison result which allows us to derive Theorem \ref{thm:line} from Theorem \ref{thm:dirichlet}, and also Theorem \ref{thm:every:dir}(\ref{t3:p2}) from Theorem \ref{thm:every:dir}(\ref{t3:p1}). Throughout this section {\it we will assume} that $b:\R\to \R$ and $\sigma:\R\to \R$ are Lipschitz continuous, with $b$ nonnegative and $\sigma(0)=0$. We start with a {\it non-negative} initial profile $u_0\ge0$ that is continuous and supported on $[0,1]$. We will consider the solution $u$ to \eqref{eq:SHE} and the solution $v$ to \eqref{eq:SHE:D} with the same initial profile $u_0=v_0$. The white noise is the same for both equations. To be precise, the white noise for \eqref{eq:SHE:D} is the restriction of the white noise in \eqref{eq:SHE} to the spatial domain $[0,1]$. We have the following
\begin{thm}\label{thm:compare} With the above assumptions we have 
\be\label{eq:compare}\bP\Big(u_t(x) \ge v_t(x) \text{ for all } t\ge 0, \, x\in [0,1]\Big)=1.\ee
Let $v^{(B)}$ be the solution to \eqref{eq:SHE:D} but with \textnormal{periodic} or \textnormal{Neumann} boundary conditions, and the same non-negative initial profile $v_0$. Then
\be
\label{eq:compare:P:D}\bP\Big(v^{(B)}_t(x) \ge v_t(x) \text{ for all } t\ge 0, \, x\in [0,1]\Big)=1.
\ee
\end{thm}

\begin{rem}It is important that the initial profile is non-negative; otherwise, the above result need not be true.
\end{rem}
\begin{rem} The above result is stated for globally Lipschitz $b$ and $\sigma$, but will be applied later to the superlinear setting. When the result is used in later sections (see for example Section \ref{sec:line}), we will use localization techniques to approximate the solution by one with globally Lipschitz coefficients.
\end{rem}

We will prove \eqref{eq:compare}. It will be clear that the same line of argument can be used to obtain \eqref{eq:compare:P:D}. To prove \eqref{eq:compare}, we first obtain a similar result for a system of interacting SDEs on the one-dimensional lattice. Let $U$ be the solution to 
\be\tag{ISDE}
\label{eq:ISDE}
dU_t(x) =(\mathscr{L}U_t)(x) dt + b(U_t(x)) dt + \sigma(U_t(x)) dB_t(x),\quad x\in \Z,
\ee
with $\mathscr{L}$ being the generator of a rate-one continuous-time simple symmetric random walk $X$, and $B_t(x),\, x\in \Z$ a sequence of independent standard Brownian motions. The initial profile $U_0(x),\, x\in\Z$ is nonnegative.
Fix $L\ge 3$. Assume that the initial profile $U_0$ is $0$ outside of $[1, L-1]$ for a fixed $L\ge 3$. We also consider $V$, the solution to
\be\tag{ISDE:D}
\label{eq:ISDE:D}
dV_t(x) =(\mathscr{L}\,V_t)(x) dt + b(V_t(x)) dt + \sigma(V_t(x)) dB_t(x),\quad x\in \{0,1\cdots, L\},
\ee
with {\it Dirichlet} boundary conditions $V_t(0) =V_t(L)=0$, and initial profile $V_0(x)=U_0(x),\, x\in \{0,1,\cdots, L\}$. Denote by $G^{(D)}$ the Dirichlet heat kernel:
\be \label{eq:G:D} G^{(D)}(t; x, y) = P_x\Big(X_t=y, \, X \text{ does not hit } \{0, L\}\text{ till time } t\Big).\ee
In particular $G^{(D)}(t; x, y)=0$ if $x\in [1, L-1]^c$ or $y\in[1, L-1]^c$.

 Our assumptions guarantee the existence and uniqueness of solutions to \eqref{eq:ISDE} and \eqref{eq:ISDE:D}. Moreover the solutions remain nonnegative, and $U_t(x)$ and $V_t(x)$ are continuous in $t$ for all $x\in \{0,1,\cdots, L\}$. We have
\begin{prop}\label{prop:u:ut} We have 
\[ \bP\Big(U_t(x) \ge V_t(x) \text{ for all } t\ge 0,\; x\in\{0,1,\cdots, L\}\Big) =1. \]
\end{prop}

We prove this using the {\it Alternating Process} defined in \cite{jose-ovha}: Consider a sequence of processes $\widetilde V^{(n)}\equiv \widetilde V $(we remove the superscript for ease of notation)  which will be an approximation to $V$. The initial profile $\widetilde V_0 = V_0$. Till time $\frac1n-$ 
\[ d\widetilde V_t(x) = b\left(\widetilde V_t(x)\right) dt + \sigma\left(\widetilde V_t(x)\right) dB_t(x),\,\quad  x\in \{0, \cdots, L\},\, 0\le t< \frac1n. \]
At time $\frac1n$
\[ \widetilde V_{\frac1n}(x) = \sum_{m=0}^{L} G^{(D)}\left(\frac1n; x, m \right) \widetilde V_{\frac1n-}(m).\]
Again from time $\frac1n$ to $\frac2n-$ we have independent SDEs: 
\[ d\widetilde V_t(x) = b\left(\widetilde V_t(x)\right) dt + \sigma\left(\widetilde V_t(x)\right) dB_t(x),\,\quad  x\in \{0, \cdots, L\},\, \frac1n\le t< \frac2n, \]
followed by 
\[\widetilde V_{\frac2n}(x) = \sum_{m=0}^{L} G^{(D)}\left(\frac1n; x, m \right) V_{\frac2n-}(m).\]
The process $\widetilde V$ is thus defined for all time. 
\begin{lem}\label{lem:approx:V}
The process $\left\{\widetilde V_t^{(n)}(x), \, x=0,1,\cdots, L\right\}_{0\le t\le T}$ converges in distribution as $n\to \infty$ to the process $\left\{V_t(x), \, x=0,1,\cdots, L\right\}_{0\le t\le T}$.
\end{lem}
The proof of this follows as in Proposition 2.3 of \cite{jose-ovha} and is left to the reader. 

Similarly we can define a sequence of Alternating processes $\widetilde U^{(n)}\equiv \widetilde U$ which are approximations of $U$. Denote 
\begin{align}\label{eq:G:Free}
    G(t; x, y)=P_x(X_t=y),\quad x, y \in \Z
\end{align}
 Till time $\frac1n-$ 
\[ d\widetilde U_t(x) = b\left(\widetilde U_t(x)\right) dt + \sigma\left( \widetilde U_t(x)\right) dB_t(x),\,\quad  x\in \Z,\, 0\le t< \frac1n. \]
At time $\frac1n$
\[  \widetilde U_{\frac1n}(x) = \sum_{m\in \Z} G\left(\frac1n; x, m \right) \widetilde U_{\frac1n-}(m).\]
Again from time $\frac1n$ to $\frac2n-$ we have independent SDEs 
\[ d \widetilde U_t(x) = b\left( \widetilde U_t(x)\right) dt + \sigma\left( \widetilde U_t(x)\right) dB_t(x),\,\quad  x\in \Z,\, \frac1n\le t< \frac2n, \]
followed by 
\[ \widetilde U_{\frac2n}(x) = \sum_{m\in\Z} G\left(\frac1n; x, m \right) \widetilde U_{\frac2n-}(m).\]
The process $\widetilde U$ is thus defined for all time. 
We have
\begin{lem}[Proposition 2.3 in \cite{jose-ovha}]\label{lem:approx:U}
The process $\left\{\widetilde U_t^{(n)}(x), \, x=0,1,\cdots, L\right\}_{0\le t\le T}$ converges in distribution as $n\to \infty$ to the process $\left\{ U_t(x), \, x=0,1,\cdots, L\right\}_{0\le t\le T}$.
\end{lem}

Now we show that $\widetilde V_t(x) \le \widetilde U_t(x)$ for all $t\ge 0$ and $x\in \{0,1,\cdots, L\}$. Indeed till time $\frac1n-$ both $\widetilde V$ and $\widetilde U$ are the same. At time $\frac1n$
\begin{align*}
\widetilde V_{\frac1n}(x) &= \sum_{m=0}^{L} G^{(D)}\left(\frac1n; x, m \right) \widetilde V_{\frac1n-}(m) \\
&\le \sum_{m=0}^{L} G^{(D)}\left(\frac1n; x, m \right) \widetilde U_{\frac1n-}(m) \\
&\le  \sum_{m=0}^L G\left(\frac1n; x, m \right) \widetilde U_{\frac1n-}(m)\\
& \le \sum_{m\in \Z} G\left(\frac1n; x, m \right) \widetilde U_{\frac1n-}(m) \\
&=  \widetilde U_{\frac1n}(x).
\end{align*}
The third line is justified for two reasons: 1) $G^{(D)} \le G$ (compare \eqref{eq:G:D} and \eqref{eq:G:Free}), and 2) the nonnegativity of the initial data implies $U_{\frac1n-}(m)\ge0$ by the comparison principle for SDEs. In the next stage $[\frac1n, \frac2n-)$, we have $\widetilde V_t(x) \le \widetilde U_t(x)$ simply by the pathwise comparison of independent SDEs. At time $\frac2n$ the above argument can again be applied to show that the ordering is preserved for all time. An application of Lemmas \ref{lem:approx:V} and \ref{lem:approx:U} then proves Proposition \ref{prop:u:ut}. \qed

\subsection{Approximations of Stochastic Heat Equation} The next step in the proof of Theorem \ref{thm:compare} is to approximate $u$ (resp. $v$) by interacting SDEs $U^{(\ve)}$ (resp. $V^{(\ve)}$), to be defined later, on a fine lattice $\ve \Z$. 

Let us first focus on the approximation of $u$. This approximation has been carried out in \cite{jose-khos-muel} and \cite{foon-jose-li} in the absence of a drift term. Here we prove the same result when a drift is included and the initial profile $u_0$ is {\it bounded}. The method closely follows that of \cite{jose-khos-muel} and \cite{foon-jose-li}.

Let 
\[ p_t(x) =\frac{1}{\sqrt{2\pi t}} \exp\left(-\frac{x^2}{2t}\right)\] 
be the Gaussian density. The mild solution of \eqref{eq:SHE} satisfies
\begin{align} \label{u_t:x}
\begin{split}
    u_t(x) &= (p_t * u_0)(x) + \int_0^t\int_\R p_{t-s}(y-x) b(u_s(y)) dyds + \int_0^t\int_\R p_{t-s}(y-x) \sigma(u_s(y)) W(dyds) \\
    &=: (p_t * u_0)(x) + \mathcal D(t,x) + \mathcal N(t,x).
\end{split}
\end{align}

Consider next the interacting diffusions $\left\{U^{(\varepsilon)}_t(x)\right\}_{x\in\varepsilon\Z,\, t\ge0}$ on the lattice $\varepsilon \Z$ which solves 
\begin{multline} \label{eq:U_t:x}
    dU^{(\varepsilon)}_t(x) =\frac{1}{2\ve^2}\sum_{y:|y-x|=\varepsilon}\left[U^{(\ve)}_t(y)-U^{(\ve)}_t(x)\right]\, dt + b\left(U^{(\ve)}_t(x)\right) dt+\frac{\sigma\left(U^{(\ve)}_t(x)\right)}{\sqrt \ve} \,dB^{(\varepsilon)}_t(x),
\end{multline}
driven by independent standard linear Brownian motions
\bes
B^{(\varepsilon)}_t(x) = \frac{1}{\sqrt\ve}\int_0^t\int_{x}^{x+\varepsilon} W(dy ds)
\ees
and the initial profile
\be \label{Ini:p:ISDE:ep}
U^{(\varepsilon)}_0(x):=\frac{1}{\varepsilon}\int_{x}^{x+\varepsilon} u_0(y) dy.
\ee
Let $P_t^{(\varepsilon)}(x)$ denote the transition kernel of $\ve X_{\frac{t}{\ve^2}}$, where $X$ is a rate one continuous time simple symmetric random walk on the lattice $\Z$. Then mild solution of \eqref{eq:U_t:x} satisfies
\be\begin{split} \label{U_t:x}
    U^{(\varepsilon)}_t (x) &= \sum_{y\in \ve\Z}P^{(\varepsilon)}_t(x-y)U_0^{(\ve)}(y) + \sum_{y\in \ve\Z} \int_0^t P^{(\varepsilon)}_{t-s}(x-y)  b\left(U^{(\varepsilon)}_s(y)\right) ds \\
    &\hspace{3cm}+\frac{1}{\sqrt{\ve}}\sum_{y \in \ve \Z} \int_0^t P^{(\varepsilon)}_{t-s}(x-y) \sigma\left(U^{(\varepsilon)}_s(y)\right) dB_s^{(\ve)}(y) \\
    &=: \left(P^{(\varepsilon)}_t * U_0^{(\ve)}\right)(x) + \textnormal{D}_{\varepsilon}(t, x)+ \textnormal{N}_{ \varepsilon}(t, x).
\end{split}\ee

We denote $\mu$ to be the step distribution of $X$, that is $\mu(1)=\mu(-1)=\frac12$. It is easy to see that $\mu$ has the characteristic function $\hat \mu(z) = \cos z$. The characteristic function satisfies Assumptions 1.1 and 1.2 of \cite{foon-jose-li} with $\alpha=2$ and $a=2$:  it is symmetric and $\{z\in (-\pi, \pi]:\hat \mu =1\}=\{0\}$, and
\begin{align*}
\cos z &= 1-\frac12 |z|^2+ O(|z|^{4})\qquad \;\;\text{as } |z|\to 0, \\
\frac{d}{dz} \cos z & = -\frac12 \frac{d}{dz} z^2 + O(|z|^{4-k}) \qquad \text{as } |z|\to 0,
\end{align*}
for all $k \le 4$. We mention here that \cite{foon-jose-li} imposes the restriction that $0<a<1$. However, Propositions 3.3 and 3.6 there continue to hold even when $a\ge 1$ as in our case; these are the only two propositions from that paper we will need. 


For $y\in \R$ we use the notation $\overline y :=\varepsilon \left[ \frac{y}{\varepsilon} \right]$. The main aim of this subsection is to prove the following.
\begin{prop}\label{prop:approx}
    Consider the random fields $u_t(x)$ and $U_t^{(\varepsilon)}(x)$ as defined in \eqref{u_t:x} and \eqref{U_t:x}. Fix $t>0$ and $p\ge 2$. Then 
    \begin{align*}
       \sup_{x\in \ve \Z}\left\|u_t(x) - U^{(\varepsilon)}_t(x)\right\|_p^2 \lesssim \ve.
    \end{align*}
\end{prop}
Define the intermediate random fields as in \cite{jose-khos-muel},
\begin{align*} 
\begin{split}
    u^{(1,\varepsilon)}_t(x) &= (p_t * u_0)(x) +  \int_0^{(t-\varepsilon^2)\vee 0}\int_\R p_{t-s}(y-x) b(u_s(\overline y)) dy ds\\
    &\hspace{4cm}+  \int_0^{(t-\varepsilon^2)\vee 0}\int_\R p_{t-s}(y - x) \sigma(u_s(\overline y)) W(dyds) \\
    &=: (p_t * u_0)(x) + \mathcal D_{u^{(1,\varepsilon)}}(t,x) + \mathcal N_{u^{(1,\varepsilon)}}(t,x),
\end{split}
\end{align*}
\begin{align*} 
\begin{split}
    u^{(2,\varepsilon)}_t(x) &= (p_t * u_0)(x) + \int_0^{(t-\varepsilon^2)\vee 0}\int_\R p_{t-s}(\overline y-\overline x) b(u_s(\overline y)) dy ds \\
    &\hspace{4cm}+  \int_0^{(t-\varepsilon^2)\vee 0}\int_\R p_{t-s}(\overline y - \overline x) \sigma(u_s(\overline y)) W(dy ds) \\
    &=: (p_t * u_0)(x) + \mathcal D_{u^{(2,\varepsilon)}}(t,x) + \mathcal N_{u^{(2,\varepsilon)}}(t,x),
\end{split}
\end{align*}
\begin{align*} 
\begin{split}
    U^{(1,\varepsilon)}_t(x) &= (p_t * u_0)(x) +  \int_0^{(t-\varepsilon^2)\vee 0}\int_\R \frac{P^{(\varepsilon)}_{t-s}(\overline y - \overline x)}{\varepsilon} b(u_s(\tilde y)) dy ds \\
    &\hspace{4cm}+ \int_0^{(t-\varepsilon^2)\vee 0}\int_\R \frac{P^{(\varepsilon)}_{t-s}(\overline y - \overline x)}{\varepsilon} \sigma(u_s(\overline y)) W(dy ds) \\
    &=: (p_t * u_0)(x) + \mathcal D_{U^{(1,\varepsilon)}}(t,x) + \mathcal N_{U^{(1,\varepsilon)}}(t,x),
\end{split}
\end{align*}
\begin{align*} 
\begin{split}
    U^{(2,\varepsilon)}_t(x) &= (p_t * u_0)(x) + \int_0^{t}\int_\R \frac{P^{(\varepsilon)}_{t-s}(\overline y -\overline x)}{\varepsilon} b(u_s(\overline y)) dy ds \\
    &\hspace{4cm}+ \int_0^{t} \int_\R\frac{P^{(\varepsilon)}_{t-s}(\overline y - \overline x)}{\varepsilon} \sigma(u_s(\overline y)) W(dy ds) \\
    &=: (p_t * u_0)(x) +\mathcal D_{U^{(2,\varepsilon)}}(t,x) + \mathcal N_{U^{(2,\varepsilon)}}(t,x).
\end{split}
\end{align*}
To simplify the exposition we use the following definiton.
\begin{defn} \label{def:apprx}
For two random fields $Y^{(\varepsilon)}(t, x),\,  Z^{(\varepsilon)}(t, x),\, t\ge 0, \, x\in \R$, say that $Y \approx Z$ if for each $T>0$ and all $p\ge 2$
\[ \sup_{x\in \R, \, t\le T}\left\|Y^{(\ve)}(t, x)-Z^{(\ve)}(t, x)\right\|^2_p \,\lesssim\, \varepsilon 
.\]
where the constants in $\lesssim$ do not depend on $\varepsilon$. 
\end{defn}


The following is standard. 
\begin{lem}[Theorem 4.2.1 and Theorem 4.3.4 in \cite{dala-sanz}] \label{lem:mom:hold}
For each $p\ge 2$ we have the following bounds on the moments: for each fixed $T>0$
\bes \sup_{x\in \R,\, t\le T} \left\| u_t(x) \right\|_p^2 <\infty.\ees
We have the following bounds on the spatial H\"older regularity: for each fixed $t>0$
\bes  \sup_{x,\, y \in \R, \, |x-y|\le \ve} \left\| u_t(x) -u_t(y)\right\|_p^2\lesssim \ve.\ees
\end{lem}

We will first show $u \approx u^{(1,\varepsilon)} \approx u^{(1,\varepsilon)} \approx U^{(1,\varepsilon)} \approx U^{(2,\varepsilon)}$. Thanks to Lemma \ref{lem:mom:hold}, the arguments in \cite{jose-khos-muel} show that $\mathcal N \approx \mathcal N_{u^{(1,\varepsilon)}} \approx \mathcal N_{u^{(2,\varepsilon)}} \approx \mathcal N_{U^{(1,\varepsilon)}} \approx \mathcal N_{U^{(2,\varepsilon)}}$, so we do not repeat them here. A similar statement for drift terms is the content of the next lemma.

\begin{lem} \label{Apprx}
    We have $ \mathcal D \approx \mathcal D_{u^{(1,\varepsilon)}} \approx  \mathcal D_{u^{(2,\varepsilon)}} \approx  \mathcal D_{U^{(1,\varepsilon)}} \approx  \mathcal D_{U^{(2,\varepsilon)}}$.
\end{lem}
\begin{proof}
    $(i)$ Jensen's inequality and Lemma \ref{lem:mom:hold} gives
    \begin{align*}
        & \left\| \mathcal D(t,x) - \mathcal D_{u^{(\varepsilon)}}(t,x) \right\|_p^2 \\
        &\lesssim \left\|  \int_0^{(t-\varepsilon^2)\vee 0}\int_\R p_{t-s}(y-x) \left[b(u_s(y)) - b(u_s(\overline y)) \right] dy ds \right\|^2_p \\
&\hspace{5cm}+ \left\|  \int_{(t-\varepsilon^2)\vee 0}^t\int_\R p_{t-s}(y-x) b(u_s(y)) \,dy ds\right\|^2_p \\
        &\lesssim t \int_0^{t}\int_\R p_{t-s}(y-x) \left\| u_s(y) - u_s(\overline y) \right\|^2_p dy ds + \varepsilon^2 \int_{(t-\varepsilon^2)\vee 0}^t\int_\R p_{t-s}(y-x) \left\| u_s(y)  \right\|^2_p dy ds\\
        &\lesssim \varepsilon.
    \end{align*}
    This gives $\mathcal D \approx \mathcal D_{u^{(1,\varepsilon)}}$.

    $(ii)$ From Lemma \ref{lem:mom:hold} we obtain
\begin{flalign*}
    \left\|  \mathcal D_{u^{(1,\varepsilon)}}(t,x) -  \mathcal D_{u^{(2,\varepsilon)}}(t,x)\right\|_p^2 &=\left\|  \int_0^{(t-\varepsilon^2)\vee 0}\int_\R \left | p_{t-s}(y-x) - p_{t-s}(\overline y-\overline x) \right | b(u_s(\overline y)) dy ds \right\|^2_p &\\
    &\lesssim \left [  \int_{\varepsilon^2}^{t\vee \ve^2} \int_\R|p_{s}(y-x) - p_{s}(\overline y-\overline x)| dy ds \right ]^2
\end{flalign*}
We now estimate the integral inside the brackets. It is clear that $\frac{\partial}{\partial x}p_s(x) \lesssim \frac{1}{\sqrt s}p_s(\frac{x}{2})$, whence 
    \begin{flalign*}
   \int_{\varepsilon^2}^{t\vee \ve^2} \int_\R\bigg | p_{s}(y-x) - p_{s}(\tilde y-\tilde x) \bigg | dy ds&=  \int_{\varepsilon^2}^{t\vee \ve^2}\int_\R\left | \int_{y-x}^{\overline y-\overline x} \frac{\partial}{\partial z}p_s(z) dz \right | dsdy &\\
    &\lesssim  \int_{\varepsilon^2}^{t\vee \ve^2}\int_\R\left | \int_{y-x}^{\overline y-\overline x} \frac{1}{\sqrt s}p_s\left(\frac{z}{2}\right)dz \right | dyds &\\
    &\lesssim \int_{\varepsilon^2}^{t\vee \ve^2}\frac{1}{\sqrt s}\left[\varepsilon^2 p_s(0)+\varepsilon \int_\R p_s\left(\frac{z}{2}\right)dz \right] ds,
\end{flalign*}
by breaking the $\R$ integral into regions $|y-x|\le \varepsilon$ and its complement. The integral in the last line above is of order $\ve$, proving $\mathcal D_{u^{(1,\varepsilon)}} \approx  \mathcal D_{u^{(2,\varepsilon)}}$.

$(iii)$ Using Lipschitz continuity of $b$ and Lemma \ref{lem:mom:hold},
    \begin{flalign*}
        \left\| \mathcal D_{u^{(2,\varepsilon)}}(t,x) -\mathcal D_{U^{(1,\varepsilon)}}(t,x) \right\|_p^2 &=\left\|  \int_0^{t-\varepsilon^2}\int_\R \left\{ p_{t-s}(\overline y - \overline x) - \frac{P^{(\varepsilon)}_{t-s}(\overline y - \overline x)}{\varepsilon} \right\} b(u_s(\overline y)) dy ds  \right\|^2_p &\\
        &\lesssim \left [\int_{\varepsilon^2}^{t\vee \ve^2}\int_\R  \left| p_{s}(\overline y - \overline x) - \frac{P^{(\varepsilon)}_{s}(\overline y -\overline x)}{\varepsilon} \right| dy ds \right]^2.
    \end{flalign*}
    Using Proposition 3.3 and Proposition 3.6 from \cite{foon-jose-li} we obtain the following bound for the expression inside the square brackets :
    \begin{flalign*}
        &\lesssim \ve\int_{\varepsilon^2}^{t\vee \ve^2} \sum_{\substack{j\in\Z \\ |j\varepsilon| \le \sqrt{s}}} \left| p_{s}(j\varepsilon) - \frac{P^{(\varepsilon)}_{s}(j\varepsilon)}{\varepsilon} \right| ds + \ve\int_{\varepsilon^2}^{t\vee \ve^2} \sum_{\substack{j\in\Z \\ |j\varepsilon|>\sqrt{s}}} \left| p_{s}(j\varepsilon) - \frac{P^{(\varepsilon)}_{s}(j\varepsilon)}{\varepsilon} \right|  ds &\\
        &\lesssim \int_{\varepsilon^2}^{t\vee \ve^2} \sum_{\substack{j\in\Z \\ |j\varepsilon| \le \sqrt{s}}} \frac{\varepsilon^{3}}{s^{3/2}} ds + \int_{\varepsilon^2}^{t\vee \ve^2} \sum_{\substack{j\in\Z \\ |j\varepsilon|>\sqrt{s}}} \frac{s\varepsilon^{3}}{|j\varepsilon|^{5}} ds, \\
        &\lesssim \int_{\varepsilon^2}^{t\vee \ve^2} \frac{\varepsilon^{2}}{s} ds. 
    \end{flalign*}
This is of order $\varepsilon^2 \log \ve$, proving $\mathcal D_{u^{(1,\varepsilon)}} \approx  \mathcal D_{U^{(1,\varepsilon)}}$.

$(iv)$ Minkowski's inequality gives
    \begin{flalign*}
   \left\| \mathcal D_{U^{(1,\varepsilon)}}(t,x) -\mathcal D_{U^{(2,\varepsilon)}}(t,x) \right\|_p^2 &=\left\| \int_{(t-\varepsilon^2)\vee 0}^t\int_\R \frac{P^{(\varepsilon)}_{t-s}(\overline y - \overline x)}{\varepsilon} \, b(u_s(\tilde y)) dy ds \right\|^2_p \\
   &\lesssim \left[\int_0^{\varepsilon^2} \sum_{j\in \Z} \frac{P^{(\varepsilon)}_{s}(j\varepsilon)}{\varepsilon} ds \right]^2 &\\
    &\lesssim \varepsilon^2,
\end{flalign*}
shows $\mathcal D_{U^{(1,\varepsilon)}} \approx  \mathcal D_{U^{(2,\varepsilon)}}$ and hence the proof of the lemma is complete.
\end{proof}

We consider the difference between convolutions of initial profiles in the next lemma.
\begin{lem} \label{Apprx5}
We have 
    \begin{align*}
     \sup_{x\in \ve\Z}\left | \left(p_{t} * u_0\right)(x) - \left(P^{(\varepsilon)}_{t} * U_0^{(\ve)}\right)(x) \right| \lesssim \left(1\wedge  \frac{\varepsilon}{\sqrt t}\right).
    \end{align*}
\end{lem}
\begin{proof}
It can be inferred from calculations carried out in step $(ii)$ of Lemma \ref{Apprx} that
\[ \sup_{x\in \ve\Z}\left | \int_\R p_{t}(y -  x)u_0(y) dy -  \int_\R p_{t}(\overline y -  x)u_0(y) dy\right|\lesssim \frac{\ve}{\sqrt t}.\] 
From the definition of initial profile $U_0^{(\ve)}(x)$ given in \eqref{Ini:p:ISDE:ep} and the calculations involving local limit theorems in step $(iii)$ of Lemma \ref{Apprx} we deduce for $t\ge \varepsilon^2$,
    \begin{flalign*}
        &\sup_{x\in \ve\Z}\left | \left(p_{t} * u_0\right)(x) - \left(P^{(\varepsilon)}_{t} * U_0^{(\ve)}\right)(x) \right|  \\
        &\lesssim \frac{\ve}{\sqrt t} +\sup_{x\in \ve\Z}\left | \sum_{\tilde y\in \varepsilon\Z} \int_{\overline y}^{\overline y + \varepsilon} p_{t}(\overline y - x)u_0(y) dy - \sum_{\overline y\in \varepsilon\Z} \int_{\overline y}^{\overline y + \varepsilon} \frac{P^{(\varepsilon)}_{t}(\overline y - x)}{\varepsilon} u_0(y) dy \right| &\\
        &\lesssim\frac{\ve}{\sqrt t}+ \varepsilon\sum_{j\in \Z} \left | p_{t}(j\varepsilon) - \frac{P^{(\varepsilon)}_{t}(j\varepsilon)}{\varepsilon} \right| &\\
        &\lesssim \frac{\varepsilon}{\sqrt t} + \frac{\ve^2}{t}.
    \end{flalign*}
   Thus the lemma is proved.
\end{proof}

In the final approximation we need to fix $t>0$, rather than take $\sup_{t\le T}$ as in Lemma \ref{Apprx}. The reason for this is that the local limit theorem approximation is not good for $t\le \ve^2$. 
\begin{lem} \label{Apprx6} Fix $t>0$. Then 
\[  \sup_{x\in \ve \Z}\left\|U^{(2, \varepsilon)}_t(x) - U^{(\varepsilon)}_t(x)\right\|_p^2\lesssim \ve.\]
\end{lem}
\begin{proof}
The definitions of $U^{(\ve)}_t(x)$ and $U^{(2,\ve)}_t(x)$ give that for any $x\in \ve\Z$,
\begin{multline*}
    \left\|U^{(2, \varepsilon)}_t(x) - U^{(\varepsilon)}_t(x)\right\|_p^2 \lesssim \left|\left(p_{t} * u_0\right)(x) - \left(P^{(\varepsilon)}_{t} * U_0^{(\ve)}\right)(x) \right|^2 \\ +\left\|\text{D}_{\ve}(t,x)-\mathcal{D}_{U^{(2,\ve)}}(t,x)\right\|_p^2+\left\|\text{N}_{\ve}(t,x)-\mathcal{N}_{U^{(2,\ve)}}(t,x)\right\|_p^2.
\end{multline*}
Lemma \ref{Apprx} and the remark preceding it immediately imply that $u_t(x) \approx U_t^{(2,\ve)}(x)$. Consequently 
    \begin{flalign*}
   & \left\| \textnormal D_\ve(t,x) - \mathcal D_{U^{(2,\varepsilon)}}(t,x) \right\|_p^2 \\
    &= \left\| \int_0^t \int_\R \frac{P^{(\varepsilon)}_{t-s}(\overline y-x)}{\ve} \left[ b(u_s(\overline y)) - b(U_s^{(\varepsilon)}(\overline y)) \right] dy ds\right\|^2_p &\\
    &\lesssim t \int_0^t \int_\R\frac{P^{(\varepsilon)}_{t-s}(\overline y-x)}{\ve}\left[ \left\|u_s(\overline y) - U^{(2,\varepsilon)}_s(\overline y)\right\|^2_p + \left\|U^{(2, \varepsilon)}_s(\overline y) - U_s^{(\varepsilon)}(\overline y) \right\|^2_p \right] dy ds &\\
    &\lesssim t^2\varepsilon + t \int_0^t \int_\R\frac{P^{(\varepsilon)}_{t-s}(\overline y-x)}{\ve}\left\|U^{(2,\varepsilon)}_s(\overline y) - U_s^{(\varepsilon)}(\overline y) \right\|^2_p dy ds.
    \end{flalign*}
Similarly
 \begin{flalign*}
&\left\| \textnormal N_\ve(t,x) - \mathcal N_{U^{(2,\varepsilon)}}(t,x) \right\|_p^2 \\
&\lesssim \int_0^t \int_\R\left[ \frac{P^{(\varepsilon)}_{t-s}(\overline y-x)}{\ve}\right]^2 \left[ \left\|u_s(\overline y) - U^{(2,\varepsilon)}_s(\overline y)\right\|^2_p + \left\|U^{(2,\varepsilon)}_s(\overline y) - U_s^{(\varepsilon)}(\overline y) \right\|^2_p \right] dy ds &\\
&\lesssim C_t\varepsilon + \int_0^t \int_\R\left[ \frac{P^{(\varepsilon)}_{t-s}(\overline y-x)}{\ve}\right]^2 \left\|U^{(2,\varepsilon)}_s(\overline y) - U_s^{(\varepsilon)}(\overline y) \right\|^2_p dy ds,
\end{flalign*}
where 
\[ C_t:= \sup_{0<\ve<1}\int_0^t \int_\R\left[ \frac{P^{(\varepsilon)}_{t-s}(\overline y-x)}{\ve}\right]^2 \, dy ds.\]
The finiteness of $C_t$ follows from Lemma 3.3 of \cite{jose-khos-muel} (see also Lemma \ref{lem:Pe}).
Let us define the non-random function
\[\mathscr{S}(t):= \sup_{x\in \ve\Z} \left\|U^{(2,\varepsilon)}_t(x) - U^{(\varepsilon)}_t(x)\right\|_p^2.\]
From standard arguments (see Lemma \ref{lem:UeWe:mombd}) one obtains
\be\label{eq:S}
    \sup_{0< \ve< 1}\sup_{s\le t} \mathscr{S}(s) <\infty.
\ee
Using Lemma \ref{Apprx5} we obtain
\begin{align*}
    &\sup_{x\in \ve\Z} \left\|U^{(2,\varepsilon)}_t(x) - U^{(\varepsilon)}_t(x)\right\|_p^2 \\
    &\lesssim \left( 1\wedge \frac{\ve^2}{t}\right)+ (C_t+t^2) \varepsilon +t \int_0^t ds \,\mathscr{S}(s ) \int_\R\frac{P^{(\varepsilon)}_{t-s}(\overline y-x)}{\ve}dy  + \int_0^t ds\,\mathscr{S}(s )\int_\R\left[ \frac{P^{(\varepsilon)}_{t-s}(\overline y-x)}{\ve}\right]^2  dy. 
\end{align*}
Using $1\wedge x \le \sqrt x$ valid for all $x\ge 0$, it follows that for every fixed $T>0$, and $t\le T$ (see the proof of Lemma \ref{lem:Pe})
\be \label{eq:S:S}
 \mathscr{S}(t) \le  D_1\frac{\ve}{\sqrt t} + D_2\ve +D_3 \int_0^t ds\,\frac{\mathscr{S}(s)}{\sqrt{t-s}},
\ee
for some constants $D_1, D_2, D_3$ dependent only on $T$. 
 Denoting $g(t):= D_1\frac{\ve}{\sqrt t}+D_2\ve$ and $k^{*m}$ the $m$-fold convolution of the kernel $k(t)= t^{-\frac12}$ we obtain 
\[ \mathscr{S}(t) \le g(t) + D_2\ve \int_0^t \sum_{m=1}^{\infty} D_3^m k^{*m}(t-s)\, ds + D_1\ve\int_0^t \sum_{m=1}^{\infty} D_3^m \frac{k^{*m}(t-s)}{\sqrt s}\, ds.\]
Identifying the $m$-fold convolutions as the densities of Dirichlet distributions we obtain the bound
\[ \mathscr{S}(t) \le g(t) + D_2\ve \sum_{m=1}^{\infty} D_3^m \frac{\Gamma(\frac12)^m \Gamma(1)}{\Gamma(\frac{m}{2}+1)} t^{\frac{m}{2}}+ D_1\ve \sum_{m=1}^{\infty} D_3^m \frac{\Gamma(\frac12)^{m+1}}{\Gamma(\frac{m+1}{2})} t^{\frac{m-1}{2}}.\]
Both the series converge, and the lemma follows since $g(t)\lesssim \ve $ for fixed $t>0$. 
\end{proof}

The proof of Proposition \ref{prop:approx} is now complete from Lemma \ref{Apprx}, the remark preceding it and Lemma \ref{Apprx6}.\qed

The following proposition is as in Theorem 2.5 of \cite{jose-khos-muel}.

\begin{prop}\label{prop:U:u:prob}
Fix $0<t_0\le T$. Then for all $M>0$
\[ \sup_{t\in [t_0, T]} \sup_{\substack{x\in \ve \Z:\\ |x|\le \ve^{-M}}} \left|U_t^{(\ve)}(x) - u_t(x)\right|\stackrel{P}{\rightarrow} 0 \quad \text{ as } \ve \downarrow 0. \]
\end{prop}
\begin{proof}
As in the proof of Theorem 2.5 in \cite{jose-khos-muel} we need to bound $\left\| U_t^{(\ve)}(x) - U_s^{(\ve)}(x)\right\|_p^2$ for $ t_0\le s<t\le T$. The bounds for $\|\text{N}_{\ve}(t,x)- \text{N}_{\ve}(s,x)\|_p^2$ are as in that paper. We just need to bound $\|\text{D}_{\ve}(t,x)-\text{D}_{\ve}(s,x)\|_p^2$ and $\left|P_t^{(\ve)}*U_0^{(\ve)}(x) - P_s^{(\ve)}*U_0^{(\ve)}(x)\right|^2$. Firstly, from Lemma 3.3 of \cite{jose-ovha}
\begin{align*}
\left|P_t^{(\ve)}*U_0^{(\ve)}(x) - P_s^{(\ve)}*U_0^{(\ve)}(x)\right|  &\lesssim \sum_j \left|P(X_{\frac{t}{\ve^2}}=j) -P(X_{\frac{s}{\ve^2}}=j)  \right| \\
& \lesssim \frac{t-s}{\ve^2}.
\end{align*}
Similarly
\begin{align*}
\|\text{D}_{\ve}(t,x)-\text{D}_{\ve}(s,x)\|_p & \lesssim  \sum_{y\in \ve\Z} \int_0^s \left| P^{(\varepsilon)}_{t-r}(x-y)  - P^{(\varepsilon)}_{s-r}(x-y)\right|\cdot \left\|b\left(U^{(\varepsilon)}_r(y)\right)\right\|_p dr \\
&\qquad +\sum_{y\in \ve\Z} \int_s^t P^{(\varepsilon)}_{t-r}(x-y)  \left\|b\left(U^{(\varepsilon)}_r(y)\right)\right\|_p dr \\
&\lesssim \frac{t-s}{\ve^2}.
\end{align*}
Therefore 
\begin{align*} \left\| U_t^{(\ve)}(x) - U_s^{(\ve)}(x)\right\|_p^p&\lesssim \left(\frac{t-s}{\ve}\right)^{p/2}+\left(\frac{t-s}{\ve^2}\right)^{p} \\
&\lesssim \frac{(t-s)^{p/2}}{\ve^{2p}}.
\end{align*}
From this and Kolmogorov continuity theorem (Theorem 1.4.3 in \cite{spde-utah}) we get for fixed $p$ large enough and uniformly for $0<\ve<1$
\be \label{eq:u:k}\E\left[\sup_{\substack{t_0<s,t\le T\\|t-s|<\ve^{10}T}}\left|U^{(\ve)}_t(x)-U^{(\ve)}_s(x)\right|^p\right] \lesssim \ve^{p/2}.\ee
We also have for $ t_0\le s<t\le T$ (see Theorem 4.3.4 in \cite{dala-sanz}):
\[ \left\|u_t(x)-u_s(x)\right\|_p^p \lesssim (t-s)^{p/2}. \]
We similarly obtain using Kolmogorov continuity theorem
\be \label{eq:U:k}\E\left[\sup_{\substack{t_0<s,t\le T\\|t-s|<\ve^{10}T}}\left|u_t(x)-u_s(x)\right|^p\right] \lesssim \ve^{p/2}.\ee
Define $\mathscr{D}^{(\ve)}_t(x) = u_t(x) -U_t^{(\ve)}(x)$. As a consequence of \eqref{eq:u:k} and \eqref{eq:U:k} we obtain 
\be \label{eq:D:k}\E\left[\sup_{\substack{t_0<s,t\le T\\|t-s|<\ve^{10}T}}\left|\mathscr{D}^{(\ve)}_t(x)-\mathscr{D}^{(\ve)}_s(x)\right|^p\right] \lesssim \ve^{p/2}.\ee
Now for $\nu<1$ we obtain using Proposition \ref{prop:approx} and \eqref{eq:D:k}
\begin{align*}
& \bP\left(|\mathscr{D}^{(\ve)}_t(x)|>\ve^{\nu/2} \text{ for some } x\in \ve \Z,\, |x|\le \ve^{-M},\, t\in [t_0, T]\right)\\
&\qquad \le 2\ve^{-M-11}\sup_{x\in \ve \Z} \sup_{t\in [t_0, T]} \bP\left(|\mathscr{D}^{(\ve)}_t(x)|>\frac12\ve^{\nu/2} \right)\\
&\hspace{2cm}+ 2\ve^{-M-11}\sup_{x\in \ve \Z} \bP\left(\sup_{\substack{t_0\le s,t\le T\\ |t-s|<\ve^{10}T}}|\mathscr{D}^{(\ve)}_t(x)-\mathscr{D}^{(\ve)}_t(x)|>\frac12\ve^{\nu/2} \right) \\
&\qquad\lesssim \ve^{-M-11+p(1-\nu)/2},
\end{align*}
which tends to $0$ if we choose $p$ large enough. 
\end{proof}
We can finally give the
\begin{proof}[Proof of Theorem \ref{thm:compare}]
Let $V^{(\ve)}$ be the corresponding approximation for \eqref{eq:SHE:D} on the $\ve \Z$ lattice, similar to \eqref{eq:U_t:x}, but with Dirichlet boundary conditions at $x=0, 1$. 
Proposition \ref{prop:u:ut} implies for each $0<\ve<1$
\[ \bP\left(U_t^{(\ve)}( x)\ge V_t^{(\ve)}(x)\text{ for all } t\ge 0,\; x\in \ve \Z, \, x\in [0,1]\right) =1. \]
Theorem 3.1 of \cite{gyon-98} shows that for $0<t_0\le T$ 
\[ \sup_{t\in [t_0, T]} \sup_{\substack{x\in \ve \Z:\\ x\in [0,1]}} \left|V_t^{(\ve)}(x) - v_t(x)\right|\stackrel{P}{\rightarrow} 0 \quad \text{ as } \ve \downarrow 0. \]
From Proposition \ref{prop:U:u:prob} we obtain 
\[ \sup_{t\in [t_0, T]} \sup_{\substack{x\in \ve \Z:\\ x\in [0,1]}} \left|U_t^{(\ve)}(x) - u_t(x)\right|\stackrel{P}{\rightarrow} 0 \quad \text{ as } \ve \downarrow 0. \]
From this we can conclude 
\[ \bP\Big(u_t(x)\ge v_t(x)\text{ for all } t\ge 0,\; x\in [0,1]\Big) =1. \]
This completes the proof of \eqref{eq:compare}.

For \eqref{eq:compare:P:D} we consider $V^{(B),(\ve)}$, the corresponding approximation for \eqref{eq:SHE:D} with {\it periodic} or {\it Neumann} boundary conditions at $x=0, 1$, similar to \eqref{eq:U_t:x}. It is clear that $G^{(B)}$ the periodic/Neumann heat kernel is above the Dirichlet heat kernel $G^{(D)}$ (see \eqref{eq:G:D}). Therefore comparing the Alternating processes for $V^{(\ve)}$ and $V^{(B),(\ve)}$ we obtain 
\[ \bP\left(V^{(B), (\ve)}_t(x)\ge V_t^{(\ve)}(x)\text{ for all } t\ge 0,\; x\in [0,1]\right) =1, \]
and an approximation result as above then gives 
\[ \bP\Big(v^{(B)}_t(x)\ge v_t(x)\text{ for all } t\ge 0,\; x\in [0,1]\Big) =1. \]
This completes the proof of Theorem \ref{thm:compare}.
\end{proof}

\section{Proof of Theorem \ref{thm:dirichlet}}\label{sec:dirichlet}
An important observation is the following
\begin{lem}[Lemma 1.1 in \cite{jose-ovha}] The Osgood condition \eqref{eq:osgood} is equivalent to 
\[ \sum_{n\ge 0} \frac{2^n}{b(2^n)} <\infty.\]
\end{lem}

We denote by $p_t(x,y), \, t>0,\, x, y\in [0,1]$ the Dirichlet heat kernel on $[0,1]$. The solution to \eqref{eq:SHE:D} satisfies
\[ v_t(x) = \left(p_t*u_0\right)(x) + \int_0^t \int_0^1 p_{t-s}(x,y) \, b\left(v_s(y)\right)\, dy \, ds + \int_0^t \int_0^1 p_{t-s}(x,y) \,\sigma(v_s(y)) \,W(dy ds).\]
Suppose for the moment $ v_0(x)= 2^n$ for $x\in[\frac13,\,\frac23]$ and $v_0(x) \le 2^n$ for all $x$. Consider the process $v^{(n)}$  defined as follows:  $v^{(n)}_0= v_0$ and
\begin{equation}\label{eq:v:2} \begin{split} v^{(n)}_t(x) &=\left(p_t*v_0\right)(x) + \int_0^t \int_0^1 p_{t-s}(x,y)  b\left(v^{(n)}_s(y)\right) dy ds \\
&\qquad\qquad  + \int_0^t \int_0^1 p_{t-s}(x,y) \,\sigma\left(v^{(n)}_s(y)\wedge \mathbf{C_1}2^{n+1}\right) W(dy ds)\\
&=:\left(p_t*v_0\right)(x)+ D_n(t,x) + N_n(t,x). \end{split}\end{equation}  Let
\[ \tau_n:= \inf\left\{t\ge 0: \sup_x v^{(n)}_t(x) =\mathbf C_12^{n+1}\right\},\]
Clearly until time $\tau_n$ we have $v_t(x)=v^{(n)}_t(x)$.
We need the following
\begin{lem}[Lemma 3.4 in \cite{athr-jose-muel}]\label{lem:N:tail:2} There exist universal constants $\K_1, \K_2$ such that  for all $\alpha, \lambda>0$ and $0\le c\le 1,\, 0<\epsilon<1$,
\[ \bP \left(\sup_{\substack{0\le t\le \alpha\epsilon^4 \\x\in [c,1\wedge (c+ \epsilon^2)]}}|N_n(t,x)| > \lambda\epsilon\right)  \le \frac{\K_1}{1\wedge \sqrt\alpha}\exp\left(-  \frac{\K_2\lambda^2}{\left[\sup_{x\le C_12^{n+1}}\sigma(x)\right]^2\sqrt\alpha}\right).\]
\end{lem}
While \cite{athr-jose-muel} work with periodic boundary conditions, it is easy to see that Lemmas 3.1 and 3.3 in \cite{athr-jose-muel} needed for the proof of the above lemma continue to hold with Dirichlet boundary conditions. 

 \begin{proof}[Proof of Theorem \ref{thm:dirichlet} (\ref{state:1})] 
 For ease of exposition, let us take $a=\frac13$. It will be clear that a similar argument works for any $0<a<\frac12$. Let 
\[ t_n := \frac{2^{n+6}}{b(2^{n-4})}.\]
Recall from the introduction that $f(x)=b(x)/x$ and $\sup_{x\le X} |\sigma(x)|\le Xg(X)$. From the above lemma, we obtain for all large $n$
\begin{align*}
\bP\left(\sup_{\substack{x\in[0,1]\\ t\le t_n}}|N_n(t,x)|> 2^{n-5}\right) &\le \frac{1}{\sqrt t_n} \bP\left(\sup_{\substack{x\in[0,\sqrt t_n]\\ t\le t_n}}|N_n(t,x)|> 2^{n-5}\left(\frac{b(2^{n-4})}{2^{n+6}}\right)^{\frac14} t_n^{\frac14}\right) \\
&\le  \K_1\frac{\sqrt{b(2^{n-4})}}{\sqrt{2^{n+6}}}\exp\left(-\frac{\K_22^{2n-10} \sqrt{b(2^{n-4})}}{\left[\sup_{x\le \mathbf{C}_12^{n+1}}\sigma(x)\right]^2\sqrt{2^{n+6}}}\right) \\
&\le \frac{\K_1}{32}\sqrt{f(2^{n-4})}\exp\left(-\frac{\K_2 \sqrt{f(2^{n-4})}}{2^{17} \mathbf C_1^2g( \mathbf C_12^{n+1})^2}\right),
\end{align*}
where we use (\ref{ass:1}) of Assumption \ref{ass:D}. From (\ref{ass:2}) of Assumption \ref{ass:D} we obtain for a large choice of $k$
\begin{align}\label{eq:tail:boud}
\begin{split}
\bP\left(\sup_{\substack{x\in[0,1]\\ t\le t_n}}|N_n(t,x)|> 2^{n-5}\right) &\lesssim \sqrt{f(2^{n-4})} \left(\frac{\sqrt{f(2^{n-4})}}{g(2^{n-4})^2}\right)^{-k} \\
&\le \left(\frac{1}{f(2^{n-4})}\right)^{k\eta} \\
&\le \frac{1}{f(2^{n-4})}.
\end{split}
\end{align}
For the probability bound to be less than $1$ we need $n\ge n_0$. Note that the sum of the first two terms in \eqref{eq:v:2} is at least $2^{n-3}$ for {\it small} $t$ and $x\in [\frac13, \frac23]$. On the event
\[ A_n =\left\{\sup_{\substack{x\in[0,1]\\ t\le t_n}}|N_n(t,x)|\le 2^{n-5}\right\},\]
 we have 
\[ v^{(n)}_t(x)>2^{n-3}- 2^{n-5}\ge 2^{n-4}  , \qquad x\in \left[\frac13, \frac23\right],\; t\le t_n .\]
Let $M_t^{(n)}:=\sup_x v_t^{(n)}(x)$. On the event $A_n$ it is easily seen 
\begin{align*} M_t^{(n)}&\le 2^n + \int_0^t b\left(M_s^{(n)}\right) ds + 2^{n-5} \\
&\le 2^{n+1} + \int_0^t b\left(M_s^{(n)}\right) ds .
\end{align*}
 Therefore on the event $A_n$ we have $M_t^{(n)}\le z_t$ where $z_t$ satisfies the ODE
\be \label{eq:ode} dz_t=b(z_t) dt\ee
with $z_0=2^{n+1}$. Denote by $T_n$ the time for $z_t$ to reach $\mathbf C_12^{n+1}$.  We have by (\ref{ass:3}) of Assumption \ref{ass:D}
\bes
T_n= \int_{2^{n+1}}^{\mathbf C_12^{n+1}} \frac{dz}{b(z)} > \frac{32 \times 2^{n+1}}{b(2^{n-4})}= t_n
\ees
 for all large $n$. Consequently on the event $A_n$ we have 
\begin{align}\label{eq:v:lt:C_1} \sup_{x\in [0,1],\, t\in [0,t_n]} v^{(n)}_t(x) \le \mathbf C_12^{n+1}, \end{align}
and hence $v$ coincides with $v^{(n)}$ on this event. Moreover for any fixed time $t\le t_n$ there is a probability of at least $\frac{1}{8}$ that a Brownian motion killed at $\{0,1\}$ and started at a point $x\in [\frac13, \frac23]$ stays within the interval $[\frac13, \frac23]$ at time $t$. Therefore on the event $A_n$
\[\inf_{x\in [\frac13,\frac23]} v_{t_n}( x) \ge 0+\frac{b(2^{n-4})}{8}t_n -2^{n-5}\ge 2^{n+1}.\]
 To summarize, there is a constant $c>0$ such that with probability at least $1- \frac{c}{f(2^{n-4})}$ the following holds: starting with $v_0(x)= 2^n$ for $x\in [\frac13, \frac23]$ and $v_0(x) \le 2^n$ for $x\in [0,1]$, the process $\inf_{x\in [\frac13, \frac23]}v_t(x)$ does not hit level $2^{n-4}$ up to time $t_n$ and crosses level $2^{n+1}$ at time $t_n$. The above is valid as long as $n\ge n_0$ for some large enough $n_0$. We now apply the Markov property and Mueller's comparison principle \cite{muel} to conclude that $\sup_{x\in [\frac13, \frac23]}v_t(x)$ blows-up by time $\sum_{n\ge n_0} t_n$ with probability at least $\prod_{n=n_0}^{\infty} \left(1- \frac{c}{f(2^{n-4})}\right)$ if the process starts with the initial profile $v_0(x)\equiv 2^{n_0}\mathbf{1}_{x\in [\frac13, \frac23]}$. 
 
 It is {\it important} to note here that we can only conclude above that $\sup_{x\in[\frac13, \frac23]} v_t(x)$  explodes by time $\sum_{n\ge n_0} t_n$, and {\it not} that $\inf_{x \in [\frac13, \frac23]}v_t(x)$ explodes by that time. This is because Mueller's comparison principle only shows the domination of $v$ over the sequence $v^{(n)}$ up to the blowup time of $v$. It is quite possible that $v$ blows up at one point well before $\sum_{n\ge n_0} t_n$.   

To prove \eqref{eq:D:blow} it remains to show that for any $\delta>0$
\be \label{eq:support} \bP\left(\inf_{x\in [\frac13, \frac23]} v(\delta,x)\ge 2^{n_0}\right) >0.\ee
This follows by an application of Theorem 2.1 of \cite{ball-mill-sole}. Although the theorem assumes Neumann boundary conditions for \eqref{eq:SHE:D} it is also valid for Dirichlet boundary conditions as mentioned in the introduction of \cite{ball-mill-sole}. Indeed one can choose $h(t,x)$  of equation (0.4) of \cite{ball-mill-sole} to be such that $\partial^2_{t,x} h=C\frac{2^{n_0}}{\delta}$ for some large enough $C>0$. 
 This proves \eqref{eq:D:blow}.
\end{proof}

\begin{proof}[Proof of Theorem \ref{thm:dirichlet} (\ref{state:2})] 

Next, we show global existence with probability $1$ if \eqref{eq:osgood} fails to hold. Thanks to \eqref{eq:compare:P:D}, we will just show global existence of \eqref{eq:SHE:D} with {\it periodic} boundary conditions. Suppose $v_0\equiv 2^n$. Let ${\mathbf C_1} 2^{n+1} \le 2^{n+k_0}$ for some $k_0\in \N$ and all $n\in \N$. 
We use instead the bound \eqref{eq:tail:boud} valid for $n\ge n_1$ and choose $k$ large enough to obtain
\[ \bP\left(\sup_{\substack{x\in[0,1]\\ t\le t_n}}|N_n(t,x)|> 2^{n-5}\right)\le \left[\frac{1}{f(2^{n-4})}\right]^2.\]
We might as well assume that $b(x)\ge Cx(\log x)$ for large $x$. Indeed, if $b(x)$ grows slower than $Cx\log x$, then we simply replace $b(x)$ with $b(x)+Cx\log x$; the global existence with this drift would then necessarily imply global existence for the smaller drift.  Therefore for $n\ge n_1$
\[ \bP\left(\sup_{\substack{x\in[0,1]\\ t\le t_n}}|N_n(t,x)|> 2^{n-5}\right)\le \frac{C_2}{n^2},\]
for some $C_2>0$.

We have shown the following in the proof of Theorem \ref{thm:dirichlet} (\ref{state:1}) : if the initial profile is $2^n$ at time $0$, then on the event $A_n:=\{\sup_{x\in[0,1],  t\le t_n}|N_n(t,x)|\le 2^{n-5}\}$ it is less than $2^{n+k_0}$ everywhere till time $t_n$ (see \eqref{eq:v:lt:C_1}). 

We now apply the Markov property as follows. Suppose at time $0$ the profile is everywhere less than $2^{n_1}$. On the event $A_{n_1}$ we have $\|v_t\|_{\infty}$ is less than $2^{n+k_0}$ till time $t_{n_1}$. Now consider the evolution of the solution from time $t_{n_1}$ onwards. Let us denote the time-shift operator 
\[ \Theta_{t_{n_1}}\dot{W}(t,x) := \dot{W}(t+n_1, x),\; t\ge 0, \, x\in \R, \]
and consider the {\it time-shifted} event $\Theta_{t_{n_1}}^{-1}A_{n_1+k_0}$. On the event $A_{n_1}\cap \Theta_{t_{n_1}}^{-1}A_{n_1+k_0} $ the solution is less than $2^{n+k_0}$ till time $t_{n_1}$, and less than $2^{n+2k_0}$ in the time interval $[t_{n_1}, t_{n_1}+t_{n_1+k_0}]$. This follows by an application of Mueller's comparison theorem, as the solution in the interval $[t_{n_1}, t_{n_1}+t_{n_1+k_0}]$ remains below the solution which starts at $2^{n_1+k_0}$ at time $t_{n_1}$. Moreover $A_{n_1}$ is {\it independent} of $\Theta_{t_{n_1}}^{-1}A_{n_1+k_0}$. Proceeding inductively we obtain that on the event $ A_{n_1}\cap \Theta_{t_{n_1}}^{-1}A_{n_1+k_0} \cap \Theta_{t_{n_1+k_0}}^{-1}A_{n_1+2k_0}\cap \cdots $, the solution remains below $2^{n_1+ik_0}$ in the interval $[0, \sum_{j\le i-1}t_{n_1+j k_0}]$ for each $i$.    
Moreover the intersection of the events has probability
\[ \prod_{i=0}^{\infty}\bP\left( \sup_{\substack{x\in[0,1]\\ t\le t_{(n_1+ik_0)}}} |N_{(n_1+ik_0)}(t,x)|\le 2^{(n_1+ik_0)-5}\right) \ge \prod_{i=0}^{\infty}\left[1-\frac{C_2}{(n_1+ik_0)^2}\right]\]
The expression on the right can be made as close to $1$ as we need by choosing $n_1$ large enough. Moreover $\sum_{i=0}^{\infty} t_{n_1+ik_0}=\infty$ for any $n_1$. What we have shown is that if the initial profile is $v_0\equiv 2^{n_1}$, then the probability of global existence approaches $1$ as $n_1\to \infty$. This would, of course, show global existence for any fixed $v_0$. This completes the proof of Theorem \ref{thm:dirichlet}.
\end{proof}

\section{Proof of Theorem \ref{thm:line}}\label{sec:line}
\begin{proof}[Proof of Theorem \ref{thm:line}]
Consider equation \eqref{eq:SHE:D} on $[-2M, 2M]$ with Dirichlet boundary conditions and initial profile $v_0\le u_0$, such that $v_0$ is positive in a neighborhood of $0$ and the support of $v_0$ contained in $[-2M, 2M]$. By Mueller's comparison theorem \cite{muel} and Theorem \ref{thm:compare}, $u$ remains above $v$. This is because $u^{(J)}$ remains above $v^{(J)}$, the solution to \eqref{eq:SHE:D} with $b(\cdot)$ replaced by $b(\cdot \wedge J)$, and $v_t(x)=\lim_{J\to \infty} v^{(J)}_t(x)$ for $t<\tau_{\infty}$. Statement  \ref{2:state:1} then follows from the first statement of Theorem \ref{thm:dirichlet}.

We next prove the second statement of Theorem \ref{thm:line}. Let $0\le h_0\le 1$ be a function supported on $[-2M, 2M]$ which is positive in a neighborhood of $0$. Let $v^{(k)}$ be the solution \eqref{eq:SHE:D} on \begin{math}[4kM-2M,\, 4kM+2M]\end{math} with Dirichlet boundary conditions and initial profile $v_0^{(k)}(x)=h_0(x-4kM)$.  From Mueller's comparison theorem \cite{muel} and Theorem \ref{thm:compare}, $u\ge v^{(k)}$ for all $k$. Since the white noise in each of the intervals $[4kM-2M,\, 4kM+2M]$ are independent, the $v^{(k)}$ are independent {\it shifted} copies of $v^{(0)}$. Statement \ref{2:state:2} then follows from Statement \ref{2:state:1}. 

\end{proof}

\section{Proof of Theorem \ref{thm:every:dir}} \label{sec:every:dir}

\begin{proof}[Proof of Theorem \ref{thm:every:dir} (\ref{t3:p1})] 
Set
\begin{align*}
    M_t := \int_0^1 \sin (\pi x)\, v_t(x) dx, \qquad A_t:=\|v_t\|_\infty.
\end{align*}
We will show that 
\[ \bP\left(\sup_{t\le \delta} M_t=\infty\right)>0.\]
From the argument following \eqref{eq:support} we can conclude that for any fixed time, say $t_0$, and any large $M>0$, there is a positive probability that $M_t$ reaches $M$ by time $t_0$.

Fix $\beta>1+\theta$. Since we are interested in the positive probability of blowup by an arbitrary time $\delta>0$, by the above argument we might as well assume that  the initial profile $v_0$ satisfies
\begin{align*}
    M_0= m_n, \qquad  A_0\le a_n,
\end{align*}
where $n$ is large. Here the {\it levels} are
\begin{align*}
    m_n= \frac{3^nn^{\beta-1-\theta}}{L}, \qquad a_n=3^nn^\beta
\end{align*}
for some large $L>0$ to be fixed later. Let
\begin{align*}
    T_{n+1} := \frac{1-\frac{1}{n}}{n^{1+\theta} \log^\theta3}.
\end{align*}
Define the stopping times
\begin{align*}
    \tau_{n+1} :=\inf \left\{ t>0: M_t\ge m_{n+1} \right\}, 
\end{align*}
and
\begin{align*}
    \rho_{n+1} :=\inf \left\{ t>0: A_t\ge a_{n+1} \right\}.
\end{align*}
We will show that with high probability $M_t$ reaches $m_{n+1}$ by time $T_{n+1}$ while $A_t$ does not reach $a_{n+1}$ by that time.
\begin{prop} \label{prop:t:levels} There exists an $r>1$ such that for all large $n$,
\begin{align*}
    \bP \left ( \tau_{n+1} \le T_{n+1} <\rho_{n+1} \right) \ge 1-n^{-r}.
\end{align*}
\end{prop}
Similar to \eqref{eq:v:2}, redefine
\begin{equation}
\begin{split} v^{(n)}_t(x) &=\left(p_t*v_0\right)(x) + \int_0^t \int_0^1 p_{t-s}(x,y) b\left( v^{(n)}_s(y) \wedge a_{n+1} \right) dyds \\
&\qquad\qquad  + \int_0^t \int_0^1 p_{t-s}(x,y) \,\left(v^{(n)}_s(y)\wedge a_{n+1}\right) W(dyds)\\
&=:\left(p_t*v_0\right)(x)+ D_n(t,x) + N_n(t,x). 
\end{split}
\end{equation}
Note that $v_t=v_t^{(n)}$ as long as $t< \rho_{n+1}$. Let 
\begin{align*}
    B_{n+1} :=\log^{1+\theta}(1+a_{n+1}), 
\end{align*}
and consider the event 
\begin{align*}
    \mathscr{A}_n := \left \{ \sup_{\substack{x\in[0,1]\\ t\le T_{n+1}}}|N_n(t,x)|\le \frac{a_n}{8} \right \}.
\end{align*}
\begin{lem}\label{lem:v:inf:bd} There exist constants $\mathbf{D}_1, \, \mathbf{D}_2>0$ such that
\[ \bP(\mathscr{A}_n^c) \le \mathbf{D}_1 \exp \left ( -\mathbf{D}_2 n^{\frac{1+\theta}{2}} \right), \]
and on the event $\mathscr{A}_n$
    \bes
\begin{split} 
\|v^{(n)}_t\|_\infty &\le\left( \left\|p_t*v_0\right\|_\infty(x) + \frac{a_n}{8} \right) \exp\left( B_{n+1} t\right).
\end{split}
\ees
\end{lem}
\begin{proof}
The first assertion follows from Lemma \ref{lem:N:tail:2}.
On the event $\mathscr{A}_n$ we have
\begin{equation*}
\begin{split} v^{(n)}_t(x) &\le\left(p_t*v_0\right)(x) + B_{n+1}\int_0^t \int_0^1 p_{t-s}(x,y) v^{(n)}_s(y)  dy ds + \frac{a_n}{8}. \\
\end{split}
\end{equation*}
A standard Gronwall argument proves the lemma. 
\end{proof}
\begin{lem} \label{lem:dir:initial}
There exists a $\mathbf{D}_3>0$ such that for all $0<t\le1$
\begin{align*}
    \|p_t * v_0\|_\infty \le \min \left \{ a_n, \frac{\mathbf{D}_3m_n}{t} \right \}.
\end{align*}
\end{lem}
\begin{proof}
Clearly $\|p_t * v_0\|_\infty \le a_n$. Using the eigenvalue expansion for the Dirichlet heat kernel we obtain 
\begin{align*}
    \frac{p_t(x, y)}{ \sin (\pi y)} = 2\sum_{n=1}^\infty e^{-n^2\pi^2 t/2} \sin(n\pi x) \frac{\sin(n\pi y)}{\sin(\pi y)}.
\end{align*}
The inequality $|\sin(n\pi y)|\le n |\sin(\pi y)|$ which follows from the sine addition formula gives 
\bes
    \sup_{x, y\in[0,1]} \frac{p_t(x, y)}{ \sin (\pi y)} \lesssim \sum_{n=1}^\infty ne^{-n^2\pi^2 t/2} \lesssim \frac{1}{t}.
\ees
Therefore,
\bes
    (p_t*v_0)(x) = \int_0^1 \frac{p_t(x, y)}{ \sin (\pi y)} \sin (\pi y) v_0(y) dy \lesssim \frac{m_n}{t},
\ees
proving the lemma.
\end{proof}

We observe that $B_{n+1}T_{n+1} \to \log 3$ as $n \to \infty$, so that for large $n$
\begin{align}\label{eq:exp:T:bd}
    e^{B_{n+1}T_{n+1}} \le 4 \qquad \text{ and} \qquad e^{B_{n+1}T_{n+1}/2} \le 2.
\end{align}
Therefore from Lemma \ref{lem:v:inf:bd}, on the event $\mathscr A_n$ we have for $t\le T_{n+1}/2$
\begin{equation*}
\begin{split} 
\|v^{(n)}_t\|_\infty &\le\left( a_n + \frac{a_n}{8} \right) \exp\left( B_{n+1} T_{n+1}/2\right) \le \frac{9}{4} a_n.
\end{split}
\end{equation*}
Next note that
\bes
    \frac{\mathbf{D}_3m_n}{T_{n+1}}  \le \frac{3^nn^\beta}{8},
\ees
for {\it sufficiently large} $L$ and $n\ge2$.
Therefore from \eqref{eq:exp:T:bd}, Lemmas \ref{lem:v:inf:bd} and \ref{lem:dir:initial}, on the event $\mathscr A_n$ we have for $\frac{T_{n+1}}{2}\le t \le T_{n+1}$
\bes
\|v^{(n)}_t\|_\infty \le\left( \frac{\mathbf{D}_3m_n}{t} + \frac{a_n}{8} \right) \exp\left( B_{n+1} T_{n+1}\right) \le \frac{3}{2} a_n.\ees
Putting together both regimes, we obtain
\bes
    \|v^{(n)}_t\|_\infty \le 3a_n < a_{n+1}
\ees
on the event $\mathscr{A}_n$. From this we conclude 
\be
\begin{split}
\bP\left(T_{n+1}<\rho_{n+1}\right) &\ge \bP(\mathscr{A}_n) \\
&\ge 1- \mathbf{D}_1 \exp \left ( -\mathbf{D}_2 n^{\frac{1+\theta}{2}} \right).
\end{split}
\ee
This proves one part of Proposition \ref{prop:t:levels}.

Next, we show that $M_t$ reaches $m_{n+1}$ with very high probability by time $T_{n+1}$. Using the key identity
\begin{align*}
    \int_0^1 \sin (\pi x) p_t(x, y) dy = e^{-\frac{\pi^2t}{2}} \sin(\pi y),
\end{align*}
we obtain
\begin{align}\label{eq:M_t:integral}
\begin{split}
    M_t &= e^{-\frac{\pi^2t}{2}} M_0 + \int_0^t \int_0^1 e^{-\frac{\pi^2(t-r)}{2}} \sin(\pi y) b\left( v_r(y) \right) dyds \\
    &\qquad\qquad  + \int_0^t \int_0^1 e^{-\frac{\pi^2(t-r)}{2}} \sin(\pi y) v_r(y) W(dyds) \\
    &=: e^{-\frac{\pi^2t}{2}} M_0 + D(t) +N(t).
\end{split}
\end{align}
Let us consider
\begin{align}\label{eq:nnt}
    N_n(t) = \int_0^t \int_0^1 e^{-\frac{\pi^2(t-r)}{2}} \sin(\pi y) v_r(y) \mathbf{1}_{\{r<\rho_{n+1}\}} W(dyds),
\end{align}
whose moments we bound now.
\begin{lem} We have for all $k\ge 2$,
\begin{align*}
    \E \left [ \sup_{t\le T_{n+1}} |N_n(t)|^k \right] \lesssim m_n^k T_{n+1}^{k/4},
\end{align*}
where the constant is dependent on $k$.
\end{lem}
\begin{proof}
First of all we note that for $s<\rho_{n+1}$
\begin{align*}
    v_s(x) \le \exp\left( B_{n+1}s \right) w_s(x)
\end{align*}
where $w$ solves the linear stochastic heat equation
\begin{align*}
    \partial_t w_t(x) = \frac{1}{2} \partial_x^2 w_t(x) + w_t(x) \dot W(t, x),\qquad t\ge 0, \, x\in [0,1],
\end{align*}
with Dirichlet boundary conditions and $w_0=v_0$. Consider a new weighted kernel
\begin{align*}
    K_t(x, y) = \frac{\sin(\pi x)}{\sin(\pi y)} p_t(x, y).
\end{align*}
From \cite[Theorem 1]{male-sera}, we have
\begin{align*}
    K_t(x, y) \le \frac{\mathbf{D}_4}{\sqrt{t}} \exp \left ( -\frac{\mathbf{D}_5 |x-y|^2}{t} \right).
\end{align*}
Therefore,
\begin{align*}
    \sup_{0<y<1} \int_0^1 K_t(x, y)^2 dx \lesssim \frac{1}{\sqrt{t}}. 
\end{align*}
An application of Burkholder's inequality gives
\begin{align*}
\int_0^1 \sin^2 (\pi x) \left \| w_t(x) \right \|_k^2 dx 
    &\le 2 \int_0^1 \sin^2 (\pi x) \left (\int_0^1 p_t(x, y) v_0(y) dy \right)^2 dx \\
    &\hspace{1cm} +8k \int_0^1 \sin^2 (\pi x) \int_0^t \int_0^1 p^2_{t-s}(x, y) \left \| w_s(y) \right \|_k^2 dyds dx \\
    &= 2 \int_0^1 \left (\int_0^1 K_t(x, y) \sin (\pi y) v_0(y) dy \right)^2 dx \\
    &\hspace{1cm} +8k \int_0^1 \int_0^t \int_0^1 K^2_{t-s}(x, y) \sin^2 (\pi y) \left \| w_s(y) \right \|_k^2 dyds dx \\
    &\lesssim \frac{m_n^2}{\sqrt{t}} + k \int_0^t \frac{1}{\sqrt{t-s}} \int_0^1 \sin^2 (\pi y) \left \| w_s(y) \right \|_k^2 dyds.
\end{align*}
Gronwall's inequality thus gives
\begin{align*}
    \int_0^1 \sin^2 &(\pi x) \left \| w_t(x) \right \|_k^2 dx \le \frac{m_n^2}{\sqrt{t}}(1+k\sqrt{t}) e^{Ck^2};
\end{align*}
see the proof of Proposition 2.1 in \cite{foon-jose}. We now complete the proof of the lemma. As
\begin{align*}
    e^{\frac{\pi^2t}{2}} N_n(t) = \int_0^t \int_0^1 e^{-\frac{\pi^2r}{2}} \sin(\pi y) v_r(y) \mathbf{1}_{\{r<\rho_{n+1}\}} W(dyds)
\end{align*}
is a martingale, we obtain from Doob's inequality
\begin{align*}
    \E \left [ \sup_{t\le T_{n+1}} | e^{\frac{\pi^2 t}{2}} N_n(t)|^k \right]^{2/k} 
    & \lesssim e^{\frac{\pi^2 T_{n+1}}{2}} \left \| N_n(T_{n+1}) \right \|_k^2 \\
    & \lesssim e^{\frac{\pi^2 T_{n+1}}{2}} \int_0^{T_{n+1}} \int_0^1 \sin^2 (\pi y) \exp\left ( 2B_{n+1} T_{n+1} \right) \left \| w_s(y) \right \|_k^2 dyds \\
    &\lesssim \int_0^{T_{n+1}} \frac{m_n^2}{\sqrt{s}} ds \\
    & \lesssim m_n^2\sqrt{T_{n+1}}
\end{align*}
the second last step follows from \eqref{eq:exp:T:bd}. The constants in the inequalities above depend on $k$.
\end{proof}
Markov's inequality then gives
\begin{align*}
    \bP \left( \sup_{t\le T_{n+1}} | N_n(t)| > \frac{m_n}{n} \right) \lesssim n^{-k\left[\frac{1+\theta}{4} -1 \right]},
\end{align*}
and this bound becomes summable by choosing $k$ sufficiently large (so that $k[\frac{1+\theta}{4}-1]>1$). {\it Note} that this is where we require $\theta>3$. Let us define the event
\begin{align*}
    \mathscr B_n := \left \{ \sup_{t\le T_{n+1}} | N_n(t)| \le  \frac{m_n}{n} \right \}.
\end{align*}
Let us note that on the event $\mathscr A_n$ we have $N(t) = N_n(t)$ for $t\le T_{n+1}$. 

The proof of Proposition \ref{prop:t:levels} will be concluded once we show that on the event $\mathscr A_n \cap \mathscr B_n$, $M_t$ reaches level $m_{n+1}$ by time $T_{n+1}$. Coming back to \eqref{eq:M_t:integral}, on the event $\mathscr A_n \cap \mathscr B_n$,
\begin{align*}
    M_t &\ge e^{-\frac{\pi^2t}{2}} m_n - \frac{m_n}{n} + \int_0^t \int_0^1 e^{-\frac{\pi^2(t-r)}{2}} \sin(\pi y) b\left( v_r(y) \right) dydr.
\end{align*}
Denote
\begin{align*}
    I := \int_0^1 \sin(\pi y) dy = \frac{2}{\pi}.
\end{align*}
Jensen's inequality gives
\begin{align*}
    \frac{M_t}{I} \ge e^{-\frac{\pi^2t}{2}} \frac{m_n}{I} - \frac{m_n}{nI} + \int_0^t e^{-\frac{\pi^2(t-r)}{2}} \frac{M_r}{I} \log^{1+\theta} \left ( 1+ \frac{M_r}{I} \right ) dr.
\end{align*}
Clearly,
\begin{align*}
    M_t \ge IZ_t,
\end{align*}
where
\begin{align*}
    Z_t = e^{-\frac{\pi^2t}{2}} \frac{m_n}{I} - \frac{m_n}{nI} + \int_0^t e^{-\frac{\pi^2(t-r)}{2}}  Z_r \log^{1+\theta} \left ( 1+ Z_r \right ) dr
\end{align*}
with $Z_0 = \frac{m_n}{I} \left ( 1- \frac{1}{n} \right)$. It is easy to check that $Z_t$ satisfies the ODE
\begin{align*}
    \frac{d}{dt} Z_t = - \frac{\pi^2}{2} Z_t - \frac{\pi^2}{2} \frac{m_n}{nI} + Z_t \log^{1+\theta} \left ( 1+ Z_t \right ).
\end{align*}
We show that the time for $Z_t$ to reach $m_{n+1}/I$ is less than $T_{n+1}$ for large $n$. The time to reach $m_{n+1}/I$ is
\begin{align}\label{eq:zt:time}
  \begin{split}  &\int_{\frac{m_{n}}{I}\left ( 1- \frac{1}{n} \right)}^{\frac{m_{n+1}}{I}} \frac{dz}{ z \left[\log^{1+\theta} \left ( 1+ z \right ) - \frac{\pi^2}{2} \right] - \frac{\pi^2}{2} \frac{m_n}{nI}} \\
    &\le \sup_{z\in \left[\frac{m_{n}}{I}\left ( 1- \frac{1}{n} \right), \frac{m_{n+1}}{I} \right]} \frac{z\log^{1+\theta} z}{ z \left[\log^{1+\theta} \left ( 1+ z \right ) - \frac{\pi^2}{2} \right] - \frac{\pi^2}{2} \frac{m_n}{nI}} \int_{\frac{m_{n}}{I}\left ( 1- \frac{1}{n} \right)}^{\frac{m_{n+1}}{I}} \frac{dz}{ z \log^{1+\theta}  z }.
    \end{split}
\end{align}
Let us analyze the supremum on the above display. This is equal to 
\begin{align*}
    &\sup_{z\in \left[\frac{m_{n}}{I}\left ( 1- \frac{1}{n} \right), \frac{m_{n+1}}{I} \right]} \frac{1}{ 1  - \frac{\pi^2}{2\log^{1+\theta}(1+z)} - \frac{\pi^2}{2z\log^{1+\theta} (1+z)} \frac{m_n}{nI}} \left[\frac{\log z}{\log (1+z)}\right]^{1+\theta} \\
&\le \sup_{z\in \left[\frac{m_{n}}{I}\left ( 1- \frac{1}{n} \right), \frac{m_{n+1}}{I} \right]} \frac{1}{ 1  - \frac{\pi^2}{\log^{1+\theta}(1+z)} } \\
    &\le 1+ \frac{\mathbf{D}_6}{n^{1+\theta}},
\end{align*}
the last two inequalities hold for $n$ large enough. Next we turn our attention to the integral besides the supremum in \eqref{eq:zt:time} which is
\begin{align*}
    &\int_{\frac{m_{n}}{I}\left ( 1- \frac{1}{n} \right)}^{\frac{m_{n+1}}{I}} \frac{dz}{ z \log^{1+\theta} z}  \\
    &= \frac{1}{\theta} \left[ \frac{1}{\log^\theta \left [ \frac{m_n}{I} \left( 1-\frac{1}{n} \right) \right]} - \frac{1}{\log^\theta \left [ \frac{m_{n+1}}{I} \right]} \right] \\
    &= \frac{1}{\theta (n\log 3)^\theta} \left[ \frac{1}{\left(1+ \frac{1}{n\log 3} \log \left [ \frac{n^{\beta-1-\theta}}{LI} \left( 1-\frac{1}{n} \right) \right]\right)^\theta} - \frac{1}{\left( 1+ \frac{1}{n} + \frac{1}{n\log 3} \log \left [ \frac{(n+1)^{\beta-1-\theta}}{LI} \right] \right)^\theta} \right].
\end{align*}
Using the identity
\begin{align*}
    \frac{1}{(1+x)^\theta} = 1-\theta x + \frac{\theta(\theta+1)}{2} x^2 + \cdots,
\end{align*}
the above integral is 
\begin{align*}
    &= \frac{1}{\theta (n\log 3)^\theta} \Bigg\lbrace 1- \frac{\theta}{n\log 3} \log \left [ \frac{n^{\beta-1-\theta}}{LI} \left( 1-\frac{1}{n} \right) \right]\\
    &\hspace{4cm}+\frac{\theta(\theta+1)}{2} \left[ \frac{1}{n\log 3} \log \left ( \frac{n^{\beta-1-\theta}}{LI} \left( 1-\frac{1}{n} \right) \right) \right]^2  + \cdots \\
    &\hspace{3cm} - 1 + \theta\left(\frac{1}{n} + \frac{1}{n\log 3} \log \left [ \frac{(n+1)^{\beta-1-\theta}}{LI} \right] \right) \\
    &\hspace{4cm}-  \frac{\theta(\theta+1)}{2}\left(\frac{1}{n} + \frac{1}{n\log 3} \log \left [ \frac{(n+1)^{\beta-1-\theta}}{LI} \right] \right)^2 + \cdots \Bigg\rbrace \\
    &= \frac{1}{\theta (n\log 3)^\theta} \Bigg\lbrace - \frac{\theta}{n\log 3} \log \left( \frac{n^{\beta-1-\theta}}{LI} \right) - \frac{\theta}{n\log 3} \log \left( 1-\frac{1}{n} \right)  \\
    &\hspace{2.5cm}  +\frac{\theta(\theta+1)}{2} \frac{1}{(n\log 3)^2} \left[ \log \left ( \frac{n^{\beta-1-\theta}}{LI} \right) + \log \left( 1-\frac{1}{n} \right) \right]^2  + \cdots \\
    &\hspace{2.5cm} + \frac{\theta}{n} + \frac{\theta}{n\log 3} \log \left [ \frac{n^{\beta-1-\theta}}{LI} \right] + \frac{\theta}{n\log 3} \log \left[ \left( 1+\frac{1}{n} \right)^{\beta-1-\theta} \right] \\
    &\hspace{2.5cm} -\frac{\theta(\theta+1)}{2} \Bigg[ \frac{1}{n^2} + \frac{1}{(n\log 3)^2} \left( \log \left [ \frac{n^{\beta-1-\theta}}{LI} \right] + \log \left[ \left( 1+\frac{1}{n} \right)^{\beta-1-\theta} \right] \right)^2  \\
    &\hspace{4.5cm}  + \frac{2}{n^2\log 3} \left( \log \left [ \frac{n^{\beta-1-\theta}}{LI} \right] + \log \left[ \left( 1+\frac{1}{n} \right)^{\beta-1-\theta} \right] \right)\Bigg]  + \cdots \Bigg\rbrace.
\end{align*}
Using the expansion
\begin{align*}
    \log(1+x) = x - \frac{x^2}{2} + \dots
\end{align*}
for small $x$, we see that the above further equals
\begin{align*}
    &= \frac{1}{\theta (n\log 3)^\theta} \left[ \frac{\theta}{n} - \frac{\theta(\theta+1)}{\log 3} (\beta-1-\theta) \frac{\log n}{n^2} + O\left ( \frac{1}{n^2} \right) \right].
\end{align*}
We conclude that the time required for $Z_t$ to reach $\frac{m_{n+1}}{I}$ is less than
\begin{align*}
    &\left( 1+ \frac{\mathbf{D}_6}{n^{1+\theta}} \right) \frac{1}{\theta (n\log 3)^\theta} \left[ \frac{\theta}{n} - \frac{\theta(\theta+1)}{\log 3} \beta-1-\theta \frac{\log n}{n^2} + O\left ( \frac{1}{n^2} \right) \right] \\
    &\le \frac{1}{n^{1+\theta} \log^\theta 3} \left[ 1- \frac{(\theta+1)(\beta-1-\theta)}{\log 3} \frac{\log n}{n} + O\left ( \frac{1}{n} \right) \right] \\
    &<T_{n+1}.
\end{align*}
We have thus shown that on the event $\mathscr A_n \cap \mathscr B_n$ $Z_t$ reaches $\frac{m_{n+1}}{I}$ by time $T_{n+1}$ and so $M_t$ reaches $m_{n+1}$ by time $T_{n+1}$. As $\bP(\mathscr A_n \cap \mathscr B_n)\ge 1-n^{-r}$ for some $r>1$ this completes the proof of Proposition \ref{prop:t:levels}.\qed\\
What we have shown is that starting with $M_0=m_n$ and $A_0\le a_n$, the probability is greater than $1-n^{-r}$ that $M_t$ reaches $m_{n+1}$ by time $\approx n^{-(1+\theta)}$, while $A_t$ does not reach $a_{n+1}$. This holds for $n$ large enough, say $n\ge n_1$. Markov property and Mueller's comparison theorem then implies that $M_t$ blows up by time $\sum_{n\ge n_1} n^{-(1+\theta)}$ with probability at least $\prod_{n\ge n_0} (1-n^{-r})$. Choose $n_1$ so that the sum $\sum_{n\ge n_1} n^{-(1+\theta)} \le \delta/2$. Furthermore, the support theorem of \cite{ball-mill-sole} implies that there is a positive probability that the solution reaches a profile with $M_0\ge m_{n_1}$ and $A_0\le a_{n_1}$ by time $\delta/2$. An application of the Markov property completes the proof of Theorem \ref{thm:every:dir}(\ref{t3:p1}).
\end{proof}
\begin{proof}[Proof of Theorem \ref{thm:every:dir} (\ref{t3:p2})] 
As in the proof of Theorem \ref{thm:line}(\ref{2:state:2}) let $v^{(k)}$ denote the solution of \eqref{eq:SHE:D} on $[k, k+1]$ with Dirichlet boundary conditions, and initial profile below the constant profile $1$, and let $v^{(k, J)}$ be the solution to the same equation with $b$ replaced by $b(v \wedge J)$.

Fix $\delta>0$. Find an interval $[k, k+1]$ such that
\begin{align*}
    \int_0^1 v_t^{(k)}(y) dy \to \infty
\end{align*}
by time $\delta/2$. Let $Z(s,  y; t, x)$ be the fundamental solution of the linear SHE on $\R$:
\begin{align*}
    \partial_t w_t(x) = \frac{1}{2} \partial_x^2 w_t(x) + w_t(x) \dot W(t, x),\qquad t\ge 0, \, x\in \R,
\end{align*}
with Dirac mass as the initial profile at time $s$. By \cite[Theorem 3.1]{albe-khan-quas-2015}, $Z$ is jointly continuous in all $4$ variables and strictly positive. Consequently, given $\epsilon>0$, there exists some $r>0$ such that
\begin{align*}
    \bP \left ( \inf_{\substack{s\in[0, \delta/2], \,t\in[3\delta/4, \delta] \\ y\in[k, k+1], \,x\in[0, 1]}} Z(s, y; t, x) \le r \right) \le \epsilon.
\end{align*}
Also, stochastic Fubini implies that
\begin{align*}
    w_t(x) = \int_\R w_s(y) Z(s, y;t, x) dy.
\end{align*}
Therefore, for any $J>0$ and $s\le\delta/2$ we get
\begin{align*}
    \inf_{x\in[0,1]} u_{3\delta/4} (x) &\ge \inf_{x\in[0,1]} u_{3\delta/4}^{(J)} (x) \\
    &\ge \inf_{x\in[0,1]} \int_k^{k+1} u_{s}^{(J)} (y) Z(s, y; 3\delta/4, x) dy \\
    &\ge r \inf_{x\in[0,1]} \int_k^{k+1} v_{s}^{(k, J)} (y) dy,
\end{align*}
with probability at least $1-\epsilon$. The second inequality follows as $u^{(J)}_{3\delta/4}$ is above $w_{3\delta/4}$ with same profile $w_s=u_s^{(J)}$ at time $s$. Since $J$ is arbitrary, we conclude
\begin{align*}
    \bP \left( \inf_{x\in[0,1]} u_{3\delta/4} (x) =\infty \right) \ge 1-\epsilon.
\end{align*}
Since $\epsilon$ and $\delta$ are arbitrary, $u$ blows up instantaneously everywhere in $[0,1]$ with probability $1$. Of course one could take any other interval instead of $[0,1]$. This completes the proof of Theorem \ref{thm:every:dir}(\ref{t3:p2}).
\end{proof}

\begin{appendix}
\section*{Additional lemmas}\label{sec:appendix}
The following two lemmas were used in the proof of Lemma \ref{Apprx6}.
\begin{lem}  \label{lem:Pe}We have for any $t>0,\, \delta>0,\,0<\ve<1$ and $x\in\ve \Z$
\[ \sup_{0\le r\le t}\int_r^{r+\delta} \int_{\R}\left[ \frac{P^{(\varepsilon)}_{s}(\overline y-x)}{\ve}\right]^2 dyds \lesssim \sqrt \delta, \]
uniformly in $\varepsilon$.
\end{lem}
\begin{proof} Since $1-\hat\mu(z) \ge cz^2$ in the interval $[-\pi, \pi]$ one has
\begin{align*}
 \int_{\R}\left[ \frac{P^{(\varepsilon)}_{s}(\overline y-x)}{\ve}\right]^2 dy& = \frac{1}{\ve} \int_{-\pi}^{\pi} \exp\left(-\frac{2s}{\ve^2}(1-\hat\mu(z))\right) dz \\
&\le \frac{1}{\ve}\int_{-\pi}^{\pi} \exp\left(-\frac{2csz^2}{\ve^2}\right) dz \\
&\lesssim \frac{1}{\sqrt s}.
\end{align*}
The lemma follows. 
\end{proof}

\begin{lem} \label{lem:UeWe:mombd}
Fix $p\ge 2$ and $T>0$. Then  
    \begin{align*}
        \sup_{0<\varepsilon<1}\,\sup_{x\in\varepsilon\Z,\, 0\le t\le T}\|U_t^{(\varepsilon)}(x)\|_p<\infty \text{ and } \sup_{0<\varepsilon<1}\sup_{x\in\varepsilon\Z,\, 0\le t\le T} \|U_t^{(2, \varepsilon)}(x)\|_p <\infty.
    \end{align*}
\end{lem}
\begin{proof}
  We have the following bound using Burkholder's inequality
    \begin{align*}
        \left\|U_t^{(\varepsilon)}(x) \right\|_p^2 &\lesssim \left|\left(P_t^{(\varepsilon)}*U_0^{(\varepsilon)}\right)(x)\right|^2 + \left\| \int_0^t\sum_{y\in\varepsilon\Z} P_{t-s}^{(\varepsilon)}(y-x) b\left(U_s^{(\varepsilon)}(y)\right) ds \right\|_p^2 \\
        &\qquad+ \frac{1}{\varepsilon}\left\| \int_0^t\sum_{y\in\varepsilon\Z} P_{t-s}^{(\varepsilon)}(y-x) \sigma\left(U_s^{(\varepsilon)}(y)\right) dB^{(\varepsilon)}_s (y)\right\|_p^2 \\
        &\lesssim \left|\sum_{ y\in \varepsilon\Z} \frac{P^{(\varepsilon)}_{t}( y - x)}{\varepsilon} \int_{ y}^{y + \varepsilon} u_0(z) dz\right|^2 + t \int_0^t\sum_{y\in\varepsilon\Z} P_{t-s}^{(\varepsilon)}(y-x) \left\| U_s^{(\varepsilon)}(y) \right\|_p^2 ds \\
        &\qquad + \varepsilon\int_0^t\sum_{y\in\varepsilon\Z} \left|\frac{P_{t-s}^{(\varepsilon)}(y-x)}{\varepsilon}\right|^2 \left\|U_s^{(\varepsilon)}(y)\right\|_p^2 ds.
    \end{align*}
    Define $\mathscr{P}^{(\varepsilon)}(t) := \sup_{x\in\ve\Z} \left\|U_t^{(\varepsilon)}(x) \right\|_p^2$. It follows from standard arguments that this is finite for any fixed $\ve$. 
Since $u_0$ is bounded we can rewrite the above as   
\begin{align*}
        \mathscr{P}^{(\varepsilon)}(t) \lesssim 1 + t\int_0^t\mathscr{P}^{(\varepsilon)}(s) ds + \int_0^t \int_\R \left|\frac{P_{t-s}^{(\varepsilon)}(\overline y-x)}{\varepsilon}\right|^2 \mathscr{P}^{(\varepsilon)}(s) dyds. 
    \end{align*}
Using Lemma \ref{lem:Pe} we obtain for constants $D_1, D_2$ dependent only on $T$ 
\[\mathscr{P}^{(\ve)}(t) \le D_1+ D_2 \int_0^t\frac{\mathscr{P}^{(\ve)}(s)}{\sqrt{t-s}}\, ds.\] 
Using the argument following \eqref{eq:S:S} we obtain 
    \begin{align*}
        \sup_{0<\ve<1}\sup_{0\le t\le T}\mathscr{P}^{(\varepsilon)}(t) <\infty,
    \end{align*}
as required. The proof of the second statement is simpler. It only uses uniform bounds on $u$ from Lemma \ref{lem:mom:hold} and is left to the reader. 
\end{proof}

\begin{lem}\label{lem:b:ass} Fix $\eta\ge 0$ and $n \ge 1$. The function 
\[ b(x)=x(\log x)(\log \log x)\cdots (\underbrace{\log\log...\log}_{(n-1) \text{times}} x)(\underbrace{\log\log...\log}_{n \text{ times}} x)^{1+\eta}\]
satisfies condition (\ref{ass:3}) of Assumption \ref{ass:D}.
\end{lem}
\begin{proof} For any $\mathbf{C}_1>1$ we have for large $X$
\begin{align*}&\int_{X}^{\mathbf{C}_1 X} \frac{dz}{b(z)} \\
&\ge \frac{1}{(\log \mathbf{C}_1X)(\log \log \mathbf{C}_1X)\cdots (\underbrace{\log\log...\log}_{(n-1) \text{times}} \mathbf{C}_1X)(\underbrace{\log\log...\log}_{n \text{ times}} \mathbf{C}_1X)^{1+\eta}}\int_{X}^{\mathbf{C}_1 X}\frac{dz}{z} \\
&\ge \frac{\log \mathbf{C}_1}{(\log \mathbf{C}_1X)(\log \log \mathbf{C}_1X)\cdots (\underbrace{\log\log...\log}_{(n-1) \text{times}} \mathbf{C}_1X)(\underbrace{\log\log...\log}_{n \text{ times}} \mathbf{C}_1X)^{1+\eta}}.\end{align*}
Note that for any fixed $\mathbf C_1>1$ and any fixed $m\ge 1$
\[ \frac{(\log\log...\log \mathbf{C}_1X)}{(\log\log...\log \frac{X}{32})} \stackrel{X\to \infty}{\longrightarrow} 1,\]
where the $\log\log \cdots \log$ appears $m$ times. As a consequence if we choose $\log \mathbf{C}_1>32^2$ we have 
\[\int_{X}^{\mathbf{C}_1 X} \frac{dz}{b(z)} >  \frac{32X}{b(X/32)}\]
for all large $X$. 
\end{proof}
\end{appendix}

\subsection*{Acknowledgements} We thank the referees for their kindness and patience in evaluating the manuscript, and for their many suggestions which helped improve this paper. M.J. was partially supported by ANRF grants CRG/2023/002667, ARGM/2025/002380/MTR and a CPDA grant from the Indian Statistical Institute. We thank Michael Salins for pointing out an error in a previous version of this paper and Mohammud Foondun for useful discussions. We thank Le Chen for suggesting using the fundamental solution of the linear SHE to complete the proof of Theorem \ref{thm:every:dir}(\ref{t3:p2}).

\textbf{AI discolusure:} The proof of Theorem \ref{thm:every:dir}(\ref{t3:p1}) was constructed with the assistance of ChatGPT. The key step where we got help was in constructing the appropriate time $T_{n+1}$ (especially the correction $\frac1n$ required to make everything work nicely). We also initially were working with exponential tail estimates for $N_n(t)$ (see \eqref{eq:nnt}), but ChatGPT suggested using moment estimates instead, as in \cite{sali-zhan}. We also used it for many of the Dirichlet heat kernel estimates required in the proof of Theorem \ref{thm:every:dir}(\ref{t3:p1}). The rest of the paper was written in $2025$ {\it without} the assistance of AI.

{\small \bibliography{growthrefs2}

@article {athr-jose-muel,
    AUTHOR = {Athreya, Siva and Joseph, Mathew and Mueller, Carl},
     TITLE = {Small ball probabilities and a support theorem for the
              stochastic heat equation},
   JOURNAL = {Ann. Probab.},
  FJOURNAL = {The Annals of Probability},
    VOLUME = {49},
      YEAR = {2021},
    NUMBER = {5},
     PAGES = {2548--2572},
      ISSN = {0091-1798},
   MRCLASS = {60H15 (60G17 60G60)},
  MRNUMBER = {4317712},
MRREVIEWER = {Vivek Kumar},
       DOI = {10.1214/21-aop1515},
       URL = {https://doi.org/10.1214/21-aop1515},
}

@article {foon-jose-li,
    AUTHOR = {Foondun, Mohammud and Joseph, Mathew and Li, Shiu-Tang},
     TITLE = {An approximation result for a class of stochastic heat
              equations with colored noise},
   JOURNAL = {Ann. Appl. Probab.},
  FJOURNAL = {The Annals of Applied Probability},
    VOLUME = {28},
      YEAR = {2018},
    NUMBER = {5},
     PAGES = {2855--2895},
      ISSN = {1050-5164},
   MRCLASS = {60H15 (35K57 35R11 35R60)},
  MRNUMBER = {3847975},
MRREVIEWER = {Mar\'{\i}a J. Garrido-Atienza},
       DOI = {10.1214/17-AAP1376},
       URL = {https://doi.org/10.1214/17-AAP1376},
}

@article {jose-khos-muel,
    AUTHOR = {Joseph, Mathew and Khoshnevisan, Davar and Mueller, Carl},
     TITLE = {Strong invariance and noise-comparison principles for some
              parabolic stochastic {PDE}s},
   JOURNAL = {Ann. Probab.},
  FJOURNAL = {The Annals of Probability},
    VOLUME = {45},
      YEAR = {2017},
    NUMBER = {1},
     PAGES = {377--403},
      ISSN = {0091-1798},
   MRCLASS = {60H15 (35K57 35R60)},
  MRNUMBER = {3601652},
       DOI = {10.1214/15-AOP1009},
       URL = {http://dx.doi.org/10.1214/15-AOP1009},
}

@article{jose-ovha,
  author       = {Mathew Joseph and Shubham Ovhal},
  title        = {Instantaneous blowup for interacting SDEs with superlinear drift},
  journal      = {Electronic Journal of Probability},
  volume       = {31},
  number       = {none},
  year         = {2026},
  pages        = {1--31},
  doi          = {10.1214/26-EJP1499},
  url          = {https://projecteuclid.org/journals/electronic-journal-of-probability/volume-31/issue-none/Instantaneous-blowup-for-interacting-SDEs-with-superlinear-drift/10.1214/26-EJP1499.full}
}

@article {ball-mill-sole,
    AUTHOR = {Bally, Vlad and Millet, Annie and Sanz-Sol\'{e}, Marta},
     TITLE = {Approximation and support theorem in {H}\"{o}lder norm for
              parabolic stochastic partial differential equations},
   JOURNAL = {Ann. Probab.},
  FJOURNAL = {The Annals of Probability},
    VOLUME = {23},
      YEAR = {1995},
    NUMBER = {1},
     PAGES = {178--222},
      ISSN = {0091-1798},
   MRCLASS = {60H15 (35K10 35R60)},
  MRNUMBER = {1330767},
MRREVIEWER = {Bj\"{o}rn Schmalfuss},
       URL =
              {http://links.jstor.org/sici?sici=0091-1798(199501)23:1<178:AASTIH>2.0.CO;2-N&origin=MSN},
}

@article {bond-groi,
    AUTHOR = {Fern\'andez Bonder, Julian and Groisman, Pablo},
     TITLE = {Time-space white noise eliminates global solutions in
              reaction-diffusion equations},
   JOURNAL = {Phys. D},
  FJOURNAL = {Physica D. Nonlinear Phenomena},
    VOLUME = {238},
      YEAR = {2009},
    NUMBER = {2},
     PAGES = {209--215},
      ISSN = {0167-2789,1872-8022},
   MRCLASS = {35R60 (35K57 60H15)},
  MRNUMBER = {2516340},
       DOI = {10.1016/j.physd.2008.09.005},
       URL = {https://doi.org/10.1016/j.physd.2008.09.005},
}

@article {chen-foon-huan-sali,
    AUTHOR = {Chen, Le and Foondun, Mohammud and Huang, Jingyu and Salins,
              Michael},
     TITLE = {Global solution for superlinear stochastic heat equation on
              {$\Bbb R^d$} under {O}sgood-type conditions},
   JOURNAL = {Nonlinearity},
  FJOURNAL = {Nonlinearity},
    VOLUME = {38},
      YEAR = {2025},
    NUMBER = {5},
     PAGES = {Paper No. 055026},
      ISSN = {0951-7715,1361-6544},
   MRCLASS = {60H15},
  MRNUMBER = {4904500},
       DOI = {10.1088/1361-6544/add3b2},
       URL = {https://doi.org/10.1088/1361-6544/add3b2},
}

@article {chen-huan,
    AUTHOR = {Chen, Le and Huang, Jingyu},
     TITLE = {Superlinear stochastic heat equation on {$\Bbb{R}^d$}},
   JOURNAL = {Proc. Amer. Math. Soc.},
  FJOURNAL = {Proceedings of the American Mathematical Society},
    VOLUME = {151},
      YEAR = {2023},
    NUMBER = {9},
     PAGES = {4063--4078},
      ISSN = {0002-9939,1088-6826},
   MRCLASS = {60H15 (35K57 35R60)},
  MRNUMBER = {4607649},
MRREVIEWER = {Kunwoo\ Kim},
       DOI = {10.1090/proc/16436},
       URL = {https://doi.org/10.1090/proc/16436},
}

@book {spde-utah,
    AUTHOR = {Dalang, Robert and Khoshnevisan, Davar and Mueller, Carl and
              Nualart, David and Xiao, Yimin},
     TITLE = {A minicourse on stochastic partial differential equations},
    SERIES = {Lecture Notes in Mathematics},
    VOLUME = {1962},
   ADDRESS = {Berlin},
      YEAR = {2009},
     PAGES = {xii+216},
      ISBN = {978-3-540-85993-2},
   MRCLASS = {60-06 (35R60 60H15)},
  MRNUMBER = {1500166 (2009k:60009)},
}

@article {dala-khos-zhan,
    AUTHOR = {Dalang, Robert C. and Khoshnevisan, Davar and Zhang, Tusheng},
     TITLE = {Global solutions to stochastic reaction-diffusion equations
              with super-linear drift and multiplicative noise},
   JOURNAL = {Ann. Probab.},
  FJOURNAL = {The Annals of Probability},
    VOLUME = {47},
      YEAR = {2019},
    NUMBER = {1},
     PAGES = {519--559},
      ISSN = {0091-1798,2168-894X},
   MRCLASS = {60H15 (35B33 35B45 35K57 35R60)},
  MRNUMBER = {3909975},
MRREVIEWER = {Petru\ A.\ Cioica-Licht},
       DOI = {10.1214/18-AOP1270},
       URL = {https://doi.org/10.1214/18-AOP1270},
}

@misc{dala-sanz,
      title={Stochastic Partial Differential Equations, Space-time White Noise and Random Fields}, 
      author={Robert C. Dalang and Marta Sanz-Solé},
      year={2025},
      eprint={2402.02119},
      archivePrefix={arXiv},
      primaryClass={math.PR},
  doi       = {https://doi.org/10.1007/978-3-032-01650-8},
      url={https://arxiv.org/abs/2402.02119}, 
}

@article {foon-khos-nual,
    AUTHOR = {Foondun, Mohammud and Khoshnevisan, Davar and Nualart,
              Eulalia},
     TITLE = {Instantaneous everywhere-blowup of parabolic {SPDE}s},
   JOURNAL = {Probab. Theory Related Fields},
  FJOURNAL = {Probability Theory and Related Fields},
    VOLUME = {190},
      YEAR = {2024},
    NUMBER = {1-2},
     PAGES = {601--624},
      ISSN = {0178-8051,1432-2064},
   MRCLASS = {60H15 (60F05 60H07)},
  MRNUMBER = {4797376},
MRREVIEWER = {Leila\ Setayeshgar},
       DOI = {10.1007/s00440-024-01263-7},
       URL = {https://doi-org.libraryisikolkata.remotexs.in/10.1007/s00440-024-01263-7},
}

@article {foon-nual,
    AUTHOR = {Foondun, Mohammud and Nualart, Eulalia},
     TITLE = {The {O}sgood condition for stochastic partial differential
              equations},
   JOURNAL = {Bernoulli},
  FJOURNAL = {Bernoulli. Official Journal of the Bernoulli Society for
              Mathematical Statistics and Probability},
    VOLUME = {27},
      YEAR = {2021},
    NUMBER = {1},
     PAGES = {295--311},
      ISSN = {1350-7265,1573-9759},
   MRCLASS = {60H15},
  MRNUMBER = {4177371},
MRREVIEWER = {Leila\ Setayeshgar},
       DOI = {10.3150/20-BEJ1240},
       URL = {https://doi.org/10.3150/20-BEJ1240},
}

@article {geis-mant,
    AUTHOR = {Gei{\ss}, Christel and Manthey, Ralf},
     TITLE = {Comparison theorems for stochastic differential equations in
              finite and infinite dimensions},
   JOURNAL = {Stochastic Process. Appl.},
  FJOURNAL = {Stochastic Processes and their Applications},
    VOLUME = {53},
      YEAR = {1994},
    NUMBER = {1},
     PAGES = {23--35},
      ISSN = {0304-4149},
     CODEN = {STOPB7},
   MRCLASS = {60H15 (60H10)},
  MRNUMBER = {1290705 (95i:60050)},
MRREVIEWER = {P. Kotelenez},
       DOI = {10.1016/0304-4149(94)90055-8},
       URL = {http://dx.doi.org/10.1016/0304-4149(94)90055-8},
}

@article {gyon-98,
    AUTHOR = {Gy\"ongy, Istv\'an},
     TITLE = {Lattice approximations for stochastic quasi-linear parabolic
              partial differential equations driven by space-time white
              noise. {I}},
   JOURNAL = {Potential Anal.},
  FJOURNAL = {Potential Analysis. An International Journal Devoted to the
              Interactions between Potential Theory, Probability Theory,
              Geometry and Functional Analysis},
    VOLUME = {9},
      YEAR = {1998},
    NUMBER = {1},
     PAGES = {1--25},
      ISSN = {0926-2601},
   MRCLASS = {60H15 (35K55 35R60)},
  MRNUMBER = {1644183},
MRREVIEWER = {David Nualart},
       DOI = {10.1023/A:1008615012377},
       URL = {http://dx.doi.org/10.1023/A:1008615012377},
}

@article {muel,
    AUTHOR = {Mueller, Carl},
     TITLE = {On the support of solutions to the heat equation with noise},
   JOURNAL = {Stochastics Stochastics Rep.},
  FJOURNAL = {Stochastics and Stochastics Reports},
    VOLUME = {37},
      YEAR = {1991},
    NUMBER = {4},
     PAGES = {225--245},
      ISSN = {1045-1129},
     CODEN = {STOCBS},
   MRCLASS = {60H15 (35R60)},
  MRNUMBER = {1149348 (93e:60122)},
MRREVIEWER = {A. Badrikian},
}

@preamble{
   "\def\cprime{$'$} "
}

@article {sali-3,
    AUTHOR = {Salins, M.},
     TITLE = {Existence and uniqueness for the mild solution of the
              stochastic heat equation with non-{L}ipschitz drift on an
              unbounded spatial domain},
   JOURNAL = {Stoch. Partial Differ. Equ. Anal. Comput.},
  FJOURNAL = {Stochastic Partial Differential Equations. Analysis and
              Computations},
    VOLUME = {9},
      YEAR = {2021},
    NUMBER = {3},
     PAGES = {714--745},
      ISSN = {2194-0401,2194-041X},
   MRCLASS = {35R60},
  MRNUMBER = {4297238},
MRREVIEWER = {Eszter\ Sikolya},
       DOI = {10.1007/s40072-020-00182-7},
       URL = {https://doi.org/10.1007/s40072-020-00182-7},
}

@article {sali-4,
    AUTHOR = {Salins, M.},
     TITLE = {Existence and uniqueness of global solutions to the stochastic heat equation with super-linear drift on an unbounded spatial domain },
   JOURNAL = {Stochastics and Dynamics},
  FJOURNAL = {Stochastic Partial Differential Equations. Analysis and
              Computations},
    VOLUME = {22},
      YEAR = {2021},
       DOI = {https://doi.org/10.1142/S0219493722500149},
    NUMBER = {5},
     }

@article {sali-2,
    AUTHOR = {Salins, Michael},
     TITLE = {Global solutions for the stochastic reaction-diffusion
              equation with super-linear multiplicative noise and strong
              dissipativity},
   JOURNAL = {Electron. J. Probab.},
  FJOURNAL = {Electronic Journal of Probability},
    VOLUME = {27},
      YEAR = {2022},
     PAGES = {Paper No. 12, 17},
      ISSN = {1083-6489},
   MRCLASS = {60H15 (35R60)},
  MRNUMBER = {4372099},
MRREVIEWER = {Markus\ Kunze},
       DOI = {10.1214/22-ejp740},
       URL = {https://doi.org/10.1214/22-ejp740},
}

@article {sali,
    AUTHOR = {Salins, Michael},
     TITLE = {Global solutions to the stochastic reaction-diffusion equation
              with superlinear accretive reaction term and superlinear
              multiplicative noise term on a bounded spatial domain},
   JOURNAL = {Trans. Amer. Math. Soc.},
  FJOURNAL = {Transactions of the American Mathematical Society},
    VOLUME = {375},
      YEAR = {2022},
    NUMBER = {11},
     PAGES = {8083--8099},
      ISSN = {0002-9947,1088-6850},
   MRCLASS = {60H15},
  MRNUMBER = {4491446},
MRREVIEWER = {Markus\ Kunze},
       DOI = {10.1090/tran/8763},
       URL = {https://doi.org/10.1090/tran/8763},
}

@article {shan-zhan,
    AUTHOR = {Shang, Shijie and Zhang, Tusheng},
     TITLE = {Global well-posedness to stochastic reaction-diffusion
              equations on the real line {$\Bbb R$} with superlinear drifts
              driven by multiplicative space-time white noise},
   JOURNAL = {Electron. J. Probab.},
  FJOURNAL = {Electronic Journal of Probability},
    VOLUME = {28},
      YEAR = {2023},
     PAGES = {Paper No. 166, 29},
      ISSN = {1083-6489},
   MRCLASS = {60H15 (35K57 35R60)},
  MRNUMBER = {4674636},
MRREVIEWER = {Stefano\ Bonaccorsi},
       DOI = {10.1214/23-ejp1057},
       URL = {https://doi.org/10.1214/23-ejp1057},
}

@article {shan-zhan-2,
    AUTHOR = {Shang, Shijie and Zhang, Tusheng},
     TITLE = {Stochastic heat equations with logarithmic nonlinearity},
   JOURNAL = {J. Differential Equations},
  FJOURNAL = {Journal of Differential Equations},
    VOLUME = {313},
      YEAR = {2022},
     PAGES = {85--121},
      ISSN = {0022-0396,1090-2732},
   MRCLASS = {60H15 (35R60)},
  MRNUMBER = {4362370},
MRREVIEWER = {Jinlong\ Wei},
       DOI = {10.1016/j.jde.2021.12.033},
       URL = {https://doi.org/10.1016/j.jde.2021.12.033},
}

@incollection {wals,
    AUTHOR = {Walsh, John B.},
     TITLE = {An introduction to stochastic partial differential equations},
 BOOKTITLE = {\'{E}cole d'\'et\'e de probabilit\'es de {S}aint-{F}lour,
              {XIV}---1984},
    SERIES = {Lecture Notes in Math.},
    VOLUME = {1180},
     PAGES = {265--439},
 PUBLISHER = {Springer},
   ADDRESS = {Berlin},
      YEAR = {1986},
   MRCLASS = {60H15 (35R60 60G20 60J80)},
  MRNUMBER = {876085 (88a:60114)},
MRREVIEWER = {Luis G. Gorostiza},
}

@article{salins2025,
  author  = {Salins, M.},
  title   = {Global solutions to the stochastic heat equation with superlinear accretive reaction term and polynomially growing multiplicative white noise coefficient},
  journal = {Stochastics and Partial Differential Equations: Analysis and Computations},
  year    = {2025},
  date    = {2025-02-20},
  issn    = {2194-041X},
  doi     = {10.1007/s40072-025-00351-6},
  url     = {https://doi.org/10.1007/s40072-025-00351-6},
}

@misc{foon-khos-nual-local,
      title={On the local well-posedness of randomly forced reaction-diffusion equations with $L^2$ initial data and a superlinear reaction term}, 
      author={Mohammud Foondun and Davar Khoshnevisan and Eulalia Nualart},
      year={2025},
      eprint={2510.00214},
      archivePrefix={arXiv},
      primaryClass={math.PR},
  doi       = {https://arxiv.org/abs/2510.00214},
      url={https://arxiv.org/abs/2510.00214}, 
}

@article{male-sera,
  author  = {Małecki, Jacek and Serafin, Grzegorz},
  title   = {Dirichlet Heat Kernel for the Laplacian in a Ball},
  journal = {Potential Analysis},
  volume  = {52},
  pages   = {545--563},
  year    = {2020},
  doi     = {10.1007/s11118-018-9750-2}
}

@article{albe-khan-quas-2015,
  author    = {Alberts, Tom and Khanin, Konstantin and Quastel, Jeremy},
  title     = {The Continuum Directed Random Polymer},
  journal   = {Journal of Statistical Physics},
  year      = {2014},
  volume    = {154},
  number    = {1-2},
  pages     = {305--326},
  doi       = {10.1007/s10955-013-0872-z}
}

@misc{sali-zhan,
      title={Positive probability of explosion for stochastic heat equation with superlinear accretive reaction term and polynomially growing multiplicative noise}, 
      author={Michael Salins and Yuyang Zhang},
      year={2026},
      eprint={2605.11319},
      archivePrefix={arXiv},
      primaryClass={math.PR},
        doi = { https://doi.org/10.48550/arXiv.2605.11319},
      url={https://arxiv.org/abs/2605.11319}, 
}

@article{foon-jose,
  author    = {Foondun, Mohammud and Joseph, Mathew},
  title     = {Remarks on non-linear noise excitability of some stochastic heat equations},
  journal   = {Stochastic Processes and their Applications},
  volume    = {124},
  number    = {10},
  pages     = {3429--3440},
  year      = {2014},
  doi       = {10.1016/j.spa.2014.04.015},
  publisher = {Elsevier}
}
\bibliographystyle{plain}}
\end{document}